\documentclass{IEEEtran}
\usepackage{cite}  % Enable this line and disable the 
\usepackage{color}

\usepackage{mathrsfs} 
\usepackage{cite} 
\usepackage{subfigure}
\usepackage[namelimits,reqno]{amsmath}
\usepackage{mathptmx} % assumes new font selection scheme installed
\usepackage{amsmath} % assumes amsmath package installed
\usepackage{amssymb}
\usepackage{amsfonts}
\usepackage{graphicx}
\usepackage{subfigure}
\usepackage{epstopdf}
\usepackage{bm}
\usepackage{float}
\usepackage{latexsym}
\usepackage{epstopdf}
\usepackage[english]{babel}
\usepackage{epsfig}
\usepackage{fancyhdr}
\usepackage{eepic}
\usepackage{pstricks}
\usepackage{textpos}
\usepackage{niceframe}
\usepackage{ntheorem}
\usepackage{fancybox}
\usepackage{hyperref}
\usepackage{tikz} \usetikzlibrary{calc} \tikzset{>=latex}
\newcommand{\beqn}{\begin{align}}
\newcommand{\eeqn}{\end{align}}

\newcommand{\beqnn}{\begin{equation*}}
\newcommand{\eeqnn}{\end{equation*}}

\newcommand{\bsubeq}{\begin{subequations}}
\newcommand{\esubeq}{\end{subequations}}

\newcommand{\bma}{\left(\begin{array}}
\newcommand{\ema}{\end{array}\right)}

\newcommand{\biz}{\begin{itemize}}
\newcommand{\eiz}{\end{itemize}}

\newcommand{\benu}{\begin{enumerate}}
\newcommand{\eenu}{\end{enumerate}}

\newcommand{\bce}{\begin{center}}
\newcommand{\ece}{\end{center}}

\newcommand{\bem}{\begin{em}}
\newcommand{\eem}{\end{em}}

\newcommand{\bpm}{\begin{pmatrix}}
\newcommand{\epm}{\end{pmatrix}}

\usepackage{color}

\allowdisplaybreaks
 
\graphicspath{{picfroude/}}
\newtheorem{propo}{Proposition}

\newtheorem{lemm}{Lemma}

\newtheorem{theorem}{Theorem}
\newtheorem*{proof}{Proof}

\begin{document}

\title{Deep Learning of Delay-Compensated Backstepping for   Reaction-Diffusion PDEs} 
\author{Shanshan Wang,  Mamadou Diagne \IEEEmembership{Member, IEEE} and Miroslav Krsti\'c \IEEEmembership{Fellow, IEEE} 
\thanks{S. Wang is with the Department of Control Science and Engineering, University of Shanghai for Science and Technology, Shanghai, P. R. China, 200093. Email: ShanshanWang@usst.edu.cn}
\thanks{M. Diagne (corresponding author) and M. Krsti\'c are with the Department of Mechanical and Aerospace Engineering, University of California San Diego, 
%9500 Gilman Dr, 
La Jolla, CA 92093. Email: mdiagne@ucsd.edu and krstic@ucsd.edu}
}
\maketitle
\begin{abstract}
Deep neural networks that approximate nonlinear function-to-function mappings, i.e., operators,  which are called DeepONet, have been demonstrated in recent articles to be capable of encoding entire PDE control methodologies, such as backstepping, so that, for each new functional coefficient of a PDE plant, the backstepping gains are obtained through a simple function evaluation. These initial results have been limited to single PDEs from a given class, approximating the solutions of only single-PDE operators for the gain kernels. In this paper we expand this framework to the approximation of multiple (cascaded) nonlinear operators. Multiple operators arise in the control of PDE systems from distinct PDE classes, such as the system in this paper: a reaction-diffusion plant, which is a parabolic PDE, with input delay, which is a hyperbolic PDE. The DeepONet-approximated nonlinear operator is a cascade/composition of the operators defined by one hyperbolic PDE of the Goursat form and one parabolic PDE on a rectangle, both of which are bilinear in their input functions and not explicitly solvable. For the delay-compensated PDE backstepping  controller, which employs the learned control operator, namely, the approximated gain kernel, we guarantee exponential stability in the $L^2$ norm of the plant state and the $H^1$ norm of the input delay state. Simulations illustrate the contributed theory. 
\end{abstract}


\section{Introduction}

By focusing on stabilization of a PDE class with input delay, we introduce the first DeePONet implementation of a PDE controller where the nonlinear operator being approximated is a composition of two operators governed by PDEs of distinct classes (hyperbolic/Goursat and parabolic/rectangular). 

\subsection{ Stabilization of PDEs with delays}

%Infinite-dimensional systems 
PDEs with time delays require sophisticated mathematical tools for designing  controllers.  Methods based on Lyapunov–Krasovskii functional (LKF) combined with Halanay inequality \cite{halanay1966method,selivanov2018delayed,fridman2014introduction},  Artstein model reduction technique in combination with predictor feedback \cite{prieur2018feedback,katz2020constructive,lhachemi2020feedback,katz2021finite,katz2021sub,lhachemi2022predictor} and the PDE backstepping predictor feedback design \cite{MKrstic2009,krstic2009dead}, have sparked significant leaps forward in the field. Our present work pertains to delay-compensated controllers procured from PDE backstepping techniques \cite{MKrstic2009}. 

Considered one of the systematic design approaches for PDE control,  backstepping was first introduced to design boundary controllers for delay-free PDEs and predictor feedback controllers for delay systems \cite{MKrstic2009,krstic2009dead}. Several research works have been published in the literature exploring the design of predictor feedback laws for linear and nonlinear finite-dimensional systems \cite{krstic2009input,bekiaris2013,diagne2017time,diagne2017compensation,bresch2009,zhu2020delay}. Recent years have seen the method applied to infinite-dimensional plants as a result of \cite{MKrstic2009,krstic2009dead}, which presents an exponentially stabilizing compensator for scalar reaction-diffusion  PDE with a long boundary input delay.  The instrumental framework resulted in the development of delay-adaptive controllers \cite{Wang2021adaptive, Wang2022} and compensators designed for input delays that vary spatially \cite{qi2020compensation} or for three-dimensional reaction-diffusion partial differential equations (PDEs) \cite{qi2019control}. Not limited to parabolic PDEs,   the approach extends to  hyperbolic PDEs subject to destabilizing boundary input delays \cite{krstic2009dead}. Likewise, an adaptive output feedback control for coupled hyperbolic systems with unknown plant's parameters and known actuator and sensor delays is developed in \cite{anfinsen2018adaptive} and trade-offs between convergence rate and delay robustness are unraveled by the authors of  \cite{auriol2018delay}.  Motivated by safe drilling operation management,   a delay-compensated control scheme  for a sandwich hyperbolic PDE in the presence of a sensor delay of arbitrary length was proposed in \cite{wang2020delay}. Linearized \textit{Aw-Rascle-Zhang} PDEs describing traffic systems as $2\times 2$ coupled hyperbolic PDEs are stabilized in \cite{qi2022delay} by designing a backstepping controller that compensates for a delayed distributed input. Moreover, the conception of bilateral boundary control for the stabilization of moving shockwaves in congested-free traffic systems that are governed by hyperbolic partial differential equations (PDEs) evolving over complementary time-varying domains, is presented in \cite{yu2020bilateral,zhao2022predictor}.  Using triggered batch least-squares identifiers (BaLSIs) \cite{karafyllis2018adaptive,karafyllis2019adaptive}, exponential regulation of both the plant and the actuator states, and exact identification of an unknown boundary input delay in finite time have been achieved in \cite{wang2023delay} for coupled hyperbolic PDE-ODE cascade systems. {A Lyapunov design delay-adaptive control law for first-order hyperbolic partial integro-differential equations (PIDEs) is proposed in \cite{qi2023delay}.} 

Despite the extensive range of these results, which 
%primarily provide in most of the cases mathematical certificates of 
guarantee 
%delay-dependent 
exponential stability, the implementation of their  intricate control laws might deter widespread adoption. We leverage the advances in deep learning by employing the DeepONet framework to facilitate the usability of PDE control in applications. Specifically, we aim to expedite the generation of predictor-feedback control gain functions by approximating them using a trained neural network.

\subsection{Machine Learning (ML) in service of established, proof-equipped model-based PDE control and delay-compensating designs}

 Methods that expedite the computation of complicated gain functions of  model-informed control laws with the help of deep learning  have recently emerged \cite{bhan2023neural,krstic2023neural}. 
 They leverage a new breakthrough in neural networks and the associated mathematical theory, referred to as DeepONet, \cite{lu2021learning}, which allows to produce arbitrarily close approximations of nonlinear operators (function-to-function maps), including solutions to partial differential equations that arise either in physical models or in PDEs that govern the gains of PDE controllers. 

% In the wake of \cite{bhan2023neural}, we are tasked with approximating functionals and operators of boundary feedback control laws using Neural Networks built upon a generalized form of 
 
 The DeepONet theory generalizes the ``universal approximation theorem'' for functions \cite{hornik1989multilayer,cybenko1989approximation} to a universal approximation  for nonlinear operators \cite{chen1995universal,lu2021learning,lu2019deeponet,li2020fourier,li2020neural}. 
 %The results consolidated in the work by \cite{lu2021learning} offers a notable advantage when approximating controller gain functions resulting from nonlinear function-to-function operators.  
 %The frameworks proposed in \cite{chen1995universal,lu2021learning} offer the capability to accurately approximate model-based controller gain functions using a neural network (NN) with a single hidden layer. This NN is designed to approximate nonlinear continuous functional and operator, providing a versatile and powerful tool for expressing complex controller gain functions in a computationally efficient and precise manner.  
 For a class of PDEs equipped with a model-based (parameter functions-dependent) stabilizing control law, such as PDE backstepping, a change in the plant parameter functions only results in a re-computation of the controller gain functions through a DeepONet map of that control method learned in advance. The mapping from the functional coefficients of the plant to the controller gain functions is encoded in the neural network architecture, which executes the gain computation as a function evaluation, rather than requiring a solution to PDEs. This exceptionally attractive approach for PDE control applications, offers the advantage of retaining the theoretical guarantees of a nominal closed-loop system in its approximate counterpart for both delay-free hyperbolic \cite{bhan2023neural} and parabolic \cite{krstic2023neural} PDEs.

\subsection{Contribution of the paper}

%\textbf{\emph{Summary of our contribution.}} 
A recent step from the development of DeepONet backstepping for PDEs \cite{bhan2023neural}, \cite{krstic2023neural} to the delay-PDE structures was made in \cite{qi2023neural}. However, in \cite{qi2023neural}, both the PDE plant and its input delay dynamics are of the hyperbolic kind and, as a result, the gain kernel equation is still a single PDE, as in the initial \cite{bhan2023neural}, \cite{krstic2023neural}. In this paper we not only advance to problems with multiple kernel PDEs, but to problems where the kernel PDEs are from different PDE classes and where the control gain function is the output of a composition of nonlinear operators defined by such multiple PDEs. 

We, specifically, develop DeepONet implementations and stability guarantees for the generalized version of the backstepping design for a reaction-diffusion PDE with input delay introduced in \cite{MKrstic2009}. Both stabilizing gain kernel functions and full-state feedback control laws of a reaction-diffusion PDE with spatially varying reactivity and constant boundary input delay are learned via DeepONet. 

Two kernel PDEs, and the associated nonlinear operators, arise in this backstepping design:
\begin{itemize}
    \item The PDE for the first kernel (`{\em backstepping} kernel') is 
    \begin{itemize}
        \item in the Goursat form (on a triangular domain), 
        \item of the second-order hyperbolic type, 
        \item with the reactivity function as the operator's input. 
    \end{itemize}
    \item The PDE for the second kernel (`{\em predictor} kernel') is 
    \begin{itemize}
        \item on a rectangular domain, 
        \item of a parabolic type, given exactly by the plant's reaction-diffusion model, 
        \item with the first PDE kernel as the initial condition of the second PDE, i.e., with the first  kernel  as the input to the operator that produces the second kernel. 
    \end{itemize}
\end{itemize}
The control gain functions are obtained as an output of a composition of the two nonlinear operators whose overall input is the reactivity function. This makes not only for challenging technical developments, but %for technical developments that 
charts the path for how DeepONet implementations and theory are to be developed for more general coupled PDEs in the future. 

%\emph{\textcolor{blue}{We posit a conceptually interesting and novel problem where the coupling of distinct partial differential equation (PDE) structures emerges: a cascade of a Goursat PDE into a conventional parabolic PDE on a rectangular domain.}} In other words,  the function-to-function mapping from the spatially varying coefficient of reactivity to the pair of gain kernel functions, consists of an operator for the predictor that compensates the delay \textcolor{blue}{(solution to a conventional parabolic)}, on top of the operator for the kernel for the reaction-diffusion system \textcolor{blue}{(solution to a Goursat PDE).} \emph{\textcolor{blue}{
From the DeepONet perspective (i.e., to a researcher without interest in control but with interest in solving PDEs by machine learning), our first-of-its-kind problem setting  represents a   significant innovation. It approximates the solution of a heretofore unencountered operator, since such a combination of PDEs studied in the paper does not arise outside of the PDE {\em control} context.
%}}     

In relation to paper \cite{krstic2023neural}, the present paper provides a major, methodology-expanding advance since the kernel operator in  \cite{krstic2023neural} is fed into an even more complex nonlinear operator, defined by the reaction-diffusion (RD) system, which is not solvable explicitly since the reactivity function, which acts as a {\em multiplicative input} to the RD PDE, makes this PDE neither explicitly solvable nor linear.

%The considered problem is far from elementary because the solution of the reaction-diffusion system is not available explicitly for an arbitrary space dependent reactivity, and certainly not for the resulting predictor kernels as opposed to \cite{bhan2023neural,krstic2023neural}. For systems described by linear ordinary differential equations (ODEs) with input delay, the predictor kernels are known from matrix exponentials, which does not hold for the case under study. Constructing two single-input-two-output maps,   we developed  
With the approximation properties that we prove for the Neural Operators (NOs) governed by the kernel PDEs, we provide a mathematical certificate for the exponential stability (in $L^2$ for the plant state and in $H^1$ for the input delay state) of the approximated delay-compensated closed-loop system. 
%More precisely, exploiting the \emph{Universal Approximation Operator  Theorem}, we prove the existence of a DeepONet approximation of the gain kernel functions that enforces 
%Exponential stability bound on the $L^2$-norm of the plant state and the $H^1$-norm of actuator state. Furthermore, we establish that the DeepONet approximation of the feedback control law leads to a practical exponential stability bound of the NO approximated closed-loop system. 

\textbf{\emph{Organization of paper.}} 
Section \ref{sec2} recalls the design of an exponentially stabilizing predictor feedback control law for the reaction-diffusion PDE with boundary input delay.    Section \ref{sec3} and \ref{Stab-Deep} present the  
approximation of predictor feedback kernels operators and the stabilization under the approximate controller gain functions via DeepONet. 
Section  \ref{Simul-kernel} presents extensive simulation results.
%while Section \ref{full} shows feasibility of the approach when full-state feedback law
%is approximated via DeepONet. 
Conclusions are in Section \ref{conclude}.

\section{Delay-Compensated PDE Backstepping Design}\label{sec2}

Let us consider the scalar  reaction-diffusion PDE  with an actuator delay $D$ at its controlled boundary defined as   
\begin{align}
\label{equ-u}
u_t(x,t)=&u_{xx}(x,t)+\lambda(x) u(x,t),\\ 
\label{equ-u-bud}
u(0,t)=&0,\\ 
u(1,t)=&U(t-D),\label{equ-u-bud-U}
\end{align}
where the full-state $u(x,t), \, (x,t)\in (0,1)\times \mathbb{R}_+$ is measurable  and  the parameter  $\lambda(x)\in C^1([0,1])$ is space-varying.  Adopting  an infinite-dimensional description of the actuator state $v(x,t)=U(t+D(x-1)),$ the delayed input $U(t-D)$ can be written as an advection equation resulting into the following  PDE-PDE cascade system that is equivalent to  \eqref{equ-u}--\eqref{equ-u-bud-U} 
\begin{align}\label{equ-u-bis}
u_t(x,t)=&u_{xx}(x,t)+\lambda(x) u(x,t),\\
u(0,t)=&0,\\ 
u(1,t)=&v(0,t),\\
Dv_t(x,t)=&v_x(x,t),\label{v-pde}\\
v(1,t)=&U(t).\label{equ-ucas0}
\end{align}
Using  PDE backstepping method \cite{MKrstic2009}, one  can map  system \eqref{equ-u-bis}--\eqref{equ-ucas0} into the following exponentially stable target system 
\begin{align}
\label{equ-w0}
w_t(x,t)=&w_{xx}(x,t),\\
w(0,t)=&0,\\ 
w(1,t)=&z(0,t),\\
Dz_t(x,t)=&z_x(x,t),\\
z(1,t)=&0,\label{equ-z0}
\end{align}
where the following transformations have been employed:
\begin{align}
&w(x,t)=u(x,t)-\int_0^xk(x,y)u(y,t)\mathrm dy,\label{equ-tranwu}
\\
&z(x,t)=v(x,t)-\int_0^1\gamma(x,y)u(y,t)\mathrm dy\nonumber\\
&~~~~~~~~~~~ -D\int_0^xq(x-y)v(y,t)\mathrm dy\label{equ-tranzv}.
\end{align}
The gain kernels in \eqref{equ-tranwu}, \eqref{equ-tranzv} are governed by the set of PDEs.
For the kernel of $k(x,y)$ satisfies
\begin{align}\label{equ-kxx}
k_{xx}(x,y)=&k_{yy}(x,y)+\lambda(y) k(x,y),\quad \forall(x,y)\in \check \Omega_1,\\
k(x,x)=&-\frac{1}{2}\int_0^x\lambda(y)\mathrm dy,\label{equ-kx}\\
k(x,0)=&0,\label{equ-k0}
\end{align}
where $\check \Omega_1=\{0< y\leq x<1\}$ and $\Omega_1=\{0\leq y\leq x \leq 1\}$. The kernel $\gamma(x,y)$ satisfies 
\begin{align}
\gamma_x(x,y)=&D(\gamma_{yy}(x,y)+\lambda (y)\gamma(x,y)),\quad \forall(x,y)\in \check \Omega_2,\label{equ-gammax}\\
\gamma(x,1)=&0,\\
\gamma(x,0)=&0,\\
\gamma(0,y)=&k(1,y),\label{equ-gamma0}
\end{align}
where $\check \Omega_2=\{0\leq x\leq 1, 0<y<1\}$ and $\Omega_2=\{0\leq x\leq 1, 0\leq y\leq 1\}$. The kernel $q$ satisfies 
\begin{align}
q(x)=-\gamma_y(x,1).\label{equ-q0}
\end{align}
From \eqref{equ-ucas0}, \eqref{equ-tranzv} and \eqref{equ-z0}, the nominal  exponentially stabilizing boundary controller for the equivalent cascade system \eqref{equ-u-bis}--\eqref{equ-ucas0} is defined as follows \cite{MKrstic2009},
\begin{align}
\label{equ-U}
U(t)=\int_{0}^{1}\gamma(1,y)u(y,t)\mathrm dy+\int_0^1 q(1-y)v(y,t)\mathrm dy.
\end{align}
%where $U^{\star},$ is the nominal control law assuming that the plant spatially varying parameter $\lambda(x)$ is known. 
Consequently, a crucial theoretical question is whether one can achieve comparable stability properties when the kernel functions are replaced by neural operator approximations.  
For the paper's main result, we aim at deriving approximate exponential stability results from the gain kernel  approximations through the DeepONet universal approximation theorem \cite{lu2021advectionDeepONet} (see Theorem 2.1). 
%when the function $\lambda(x)$ is subject to variabilities.

\begin{figure}[t]
\centering
\includegraphics[width=0.485\textwidth]{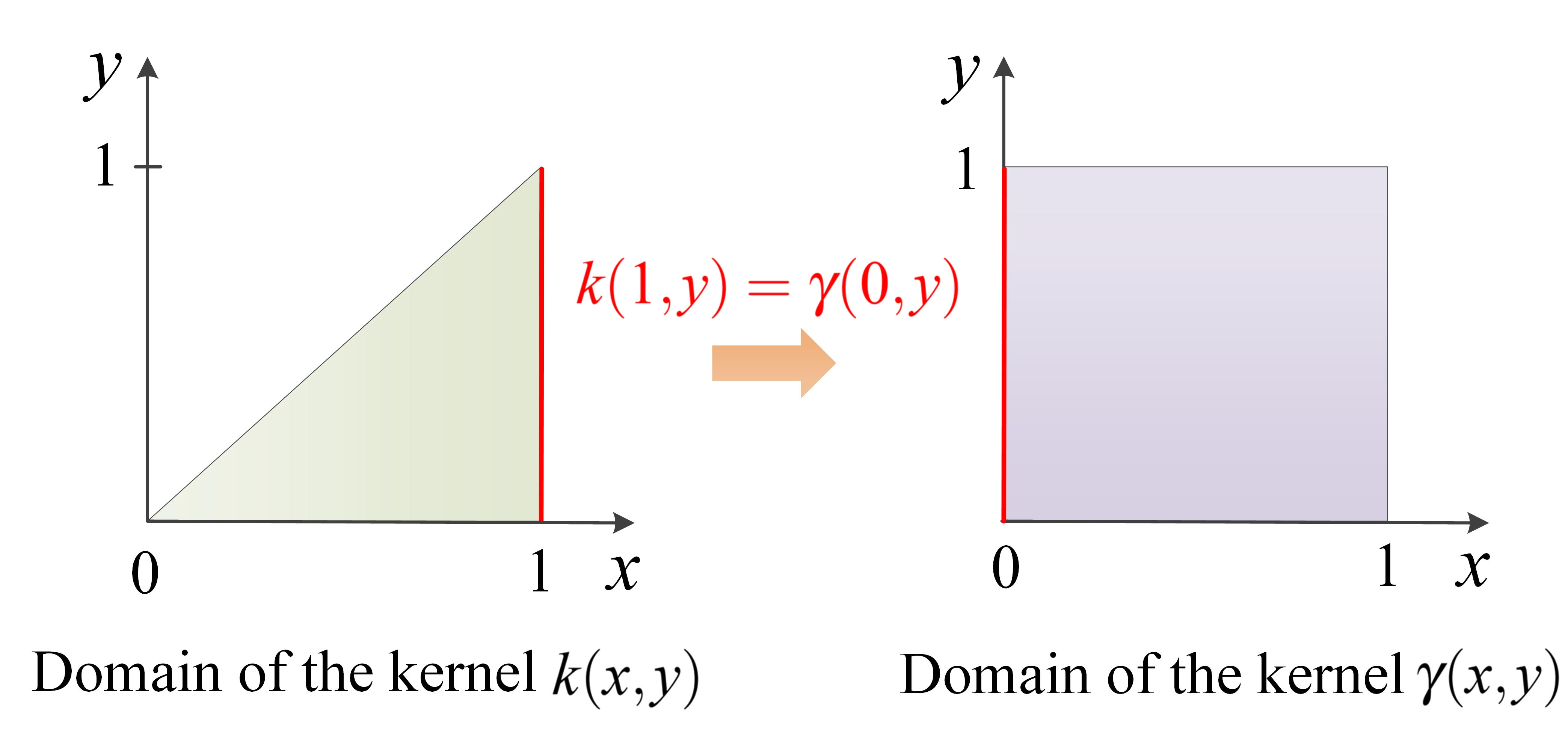}
\vspace{-0.21cm}
\caption{Triangular and  rectangular domains of the gain kernel PDEs. Left: Goursat (backsteppign) kernel. Right: reaction-diffusion (predictor) kernel. Connection: Goursat/backstepping kernel  serves as initial condition to reaction-diffusion/predictor kernel.} \label{Operator+}
\end{figure}

In the subsequent development, we treat \eqref{equ-kxx}--\eqref{equ-q0} as a cascade of two 
%single-input-two-output 
nonlinear operators:  
\begin{itemize}
    \item the ``Goursat kernel map" $\lambda \mapsto k$, whose output $k$ we shall refer to as the {\em backstepping kernel},
    \item the ``reaction-diffusion kernel map'' $k\mapsto \gamma$, whose output $\gamma$ we  refer to as the {\em predictor kernel}.
\end{itemize}
The domains of these two kernels (triangular for Goursat, and rectangular for reaction-diffusion), and their cascade connection, 
%Ultimately, The interconnection from the Goursat PDE into the conventional parabolic PDE on a rectangular domain is 
are depicted in Figure \ref{Operator+}. 

For the implementation of the controller in \eqref{equ-U}, whose gains are $\gamma(1,y)$ and $q(1-y)= -\gamma_y(1-y,1)$, the operator of interest is the composition operator $\lambda \mapsto \gamma$.

\section{Accuracy of Approximation of Backstepping Kernel Operator with DeepONet }\label{sec3}
\subsection{Boundedness of the gain kernel functions}

In this section we establish smoothness and furnish bounds on the kernels of the backstepping transformations. 
%To establish the \textcolor{red}{boundedness} of the kernels, we first provide estimates that assert their boundedness based on the following lemmas. 

% \newtheorem{lemm}{Lemma}
\begin{lemm}\label{lem1}
{\bf\em [bound on Goursat kernel]}
For every $\lambda\in C^1([0,1])$, the gain kernel $k(x,y)$ satisfying the PDE system \eqref{equ-kxx}--\eqref{equ-k0} has a unique $C^2(\Omega_1)$ solution with the following  property
\begin{align}
|k(x,y)|\leq \bar \lambda {\rm e}^{2\bar \lambda x},\quad\forall(x,y)\in\Omega_1,
\label{equ-k-bouded}
% &|\gamma(x,y)|\leq \bar \lambda {\rm e}^{2\bar \lambda x},\quad\forall(x,y)\in\Omega_2,\\
% &|q(x,y)|\leq \bar \lambda {\rm e}^{2\bar \lambda x},\quad\forall(x,y)\in\Omega_2,
\end{align}
where $\bar \lambda=\sup_{x\in[0,1]}|\lambda(x)|$.
\end{lemm}

\begin{proof}\em
    In \cite{1369395,krstic2023neural}.
\mbox{}\hfill$\blacksquare$
\end{proof}

% %\textbf{Proof:} 
% \begin{proof}\rm
% Lemma \ref{lem1} states the boundedness of the gain kernel function $k(x,y)$ and its proof is found in \cite{1369395,krstic2023neural}.
% \hfill$\blacksquare$
% \end{proof}

%Lemma \ref{lem2} whose proof is 
Next, we turn our attention to the backstepping kernel functions $\gamma$ and $q$. We introduce a  space ${\Upsilon}\subset C^1((0,1]\times[0,1])$ of functions $\gamma(x,y)$ satisfying $\gamma[x] \in C^2([0,1])$ for all $x\in(0,1]$, namely, functions of two variables which are differentiable at least once in the first variable and at least twice in the second variable, for all positive values of the first variable. 

The next lemma is useful for providing the bounds for the gain kernels $\gamma(x,y)$ and $q(x)=-\gamma_y(x,1)$, based on \eqref{gamma-bound-derivative}. 

\begin{lemm}\label{lem2}
{\bf\em [bound on reaction-diffusion kernel, with Goursat kernel as initial condition]}
For every $\lambda,\ k(1,\cdot)\in C^1([0,1])$, the gain kernel $\gamma(x,y)$ satisfying the PDE system \eqref{equ-gammax}--\eqref{equ-gamma0} has a unique %$C^2(\Omega_2)$ 
solution in the function space $\Upsilon$ with the following  property, for all $(x,y)\in \Omega_2$:
\begin{align}\label{gamma-bound-2}
|\gamma(x,y)|^2
\leq& {\rm e}^{2D\bar \lambda x}\int_0^1(k(1,y)^2+k_{y}(1,y)^2)\mathrm dy,\\
|\gamma_y(x,1)|^2\leq&(1+2 D\bar \lambda )\bar \lambda {\rm e}^{2D\bar \lambda x}\int_0^1k(1,y)^2\mathrm dy\nonumber\\
\nonumber&+\left(1+\bar \lambda+\frac{1}{D}\right) {\rm e}^{2D\bar \lambda x}\int_0^1k_{y}(1,y)^2\mathrm dy.\nonumber\\
 &+2 D {\rm e}^{2D\bar \lambda x}\int_0^1 k_{yy}(1,y)^2\mathrm dy.\label{gamma-bound-derivative}
\end{align}
\end{lemm}\label{lemma-2}

\begin{proof}\em 
    In Appendix \ref{appendix-A}.
\mbox{}\hfill$\blacksquare$
\end{proof}

\subsection{Approximation of the neural operators }

The controller \eqref{equ-U} employs the gain functions $\gamma(1,y)$ and $- \gamma_y(1-y,1)$, which are both obtained from the predictor kernel $\gamma$, which  is governed by PDE  \eqref{equ-gammax}--\eqref{equ-gamma0}. This PDE has $k(1,y)$ as its initial condition, while, in turn, the backstepping kernel $k$'s PDE \eqref{equ-kxx}--\eqref{equ-k0} is driven by function $\lambda$. 

In summary, to implement controller \eqref{equ-U}, one needs to generate the output function $\gamma$ of the operator composition $\lambda \mapsto k \mapsto \gamma$, i.e., of the backstepping-predictor kernel cascade, for a given input function $\lambda$. 

Hence, our controller implementation calls for a neural approximation of the composition mapping $\lambda \mapsto \gamma$. In more precise terms, since the PDE  \eqref{equ-gammax}--\eqref{equ-gamma0} depends not only on the initial condition $k(1,y)$ but also on the reactivity $\lambda(y)$, we need a neural approximation of two operators,  $\lambda \mapsto k$ and $(\lambda, k(1,\cdot))\mapsto \gamma$, as shown in Figure \ref{Operator}.

% We pursue our design by characterizing the neural operators that allow to learn the mapping between infinite-dimensional spaces using finite input-output pairs of collected data.  Since the kernel $q(x)$ satisfies $q(x)=-\gamma_y(x,1)$, which is obtained by numerical differentiation of the  gain kernel $\gamma$, our sole goal for controller implementation is to learn backstepping gain $\gamma(x,y)$.
% According to the gain kernels PDEs  \eqref{equ-gammax}--\eqref{equ-gamma0}, $\gamma(x,y)$ satisfies the following relation
% \begin{align}
% \int_0^x\gamma_s(s,y)\mathrm ds=D\int_0^x(\gamma_{yy}(s,y)+\lambda (y)\gamma(s,y))\mathrm ds,
% \end{align}
% which can be rewritten as follows by integration
% \begin{align}
% \nonumber\gamma(x,y)=&\gamma(0,y)+D\int_0^x(\gamma_{yy}(s,y)+\lambda (y)\gamma(s,y))\mathrm ds\\
% =&k(1,y)+D\int_0^x(\gamma_{yy}(s,y)+\lambda (y)\gamma(s,y))\mathrm ds.
% \label{eq52}
% \end{align}
% In our succeeding developments, the integral equation \eqref{eq52}  will be exploited to generate solutions
% for $k(x,y)$ and $\gamma(x,y)$. These solutions will be used to train the neural approximation of the operator  $\lambda \mapsto k$ and $(\lambda, k(1,\cdot))\mapsto \gamma$, as shown in Figure \ref{Operator}. 

Denote the sets of functions 
\begin{align}
&\underline K=\{k\in C^2(\Omega_1)|k(x,0)=0,\quad  \forall x\in [0,1]\},\\
&\underline \Gamma=\{\gamma\in \Upsilon %C^2(\Omega_2)
|\gamma(x,0)=\gamma(x,1)=0, \quad \forall x\in [0,1]\}, 
\end{align} 
and define the operators ${\mathcal{K}_1}: C^1[0,1]\to \underline K$ and ${\mathcal{K}_2}: C^1[0,1]\times C^1[0,1]\to \underline \Gamma$, where
\begin{eqnarray}
{\mathcal{K}_1}(\lambda)(x,y)&:= &k(x,y),\label{neur-op1}\\
{\mathcal{K}_2}(\lambda,k(1,\cdot))(x,y) &:= &\gamma(x,y),\label{neur-op2}
\end{eqnarray}
respectively. This allows to introduce the operators  ${\mathcal{M}}_1:C^1[0,1]\to \underline K\times C^1[0,1]\times \Upsilon %C^2(\Omega_2)
$  defined by 
\begin{align}
{\mathcal{M}}_1(\lambda)(x,y)\label{equ-M-1}
&:= (k(x,y),K_1(x),K_2(x,y))\,,
\end{align}
where
\begin{align}
&K_1(x)=2\frac{\mathrm d}{\mathrm dx}(k(x,x))+\lambda(x),\label{equ-k-1}\\
&K_2(x,y)=k_{xx}(x,y)-k_{yy}(x,y)-\lambda(y)k(x,y),\label{equ-k-2}
\end{align}
and  $\mathcal{M}_2:C^1[0,1]\times C^1[0,1]\to \underline\Gamma\times \Upsilon $  defined by
\begin{align}
{\mathcal{M}}_2(\lambda, k(1,\cdot))(x,y):=(\gamma(x,y),\Gamma(x,y)), \label{equ-M-2}
\end{align}
where
\begin{align}
\Gamma(x,y)=\gamma_x(x,y)-D\gamma_{yy}(x,y)-D\lambda(y) \gamma(x,y).\label{equ-gamma-1}
\end{align}

For the operators $\mathcal{M}_1$ and $\mathcal{M}_2$, we establish the following approximation theorem using \cite[Thm. 2.1]{lu2021advectionDeepONet}.

\begin{figure}[t]
\centering
\includegraphics[width=0.38\textwidth]{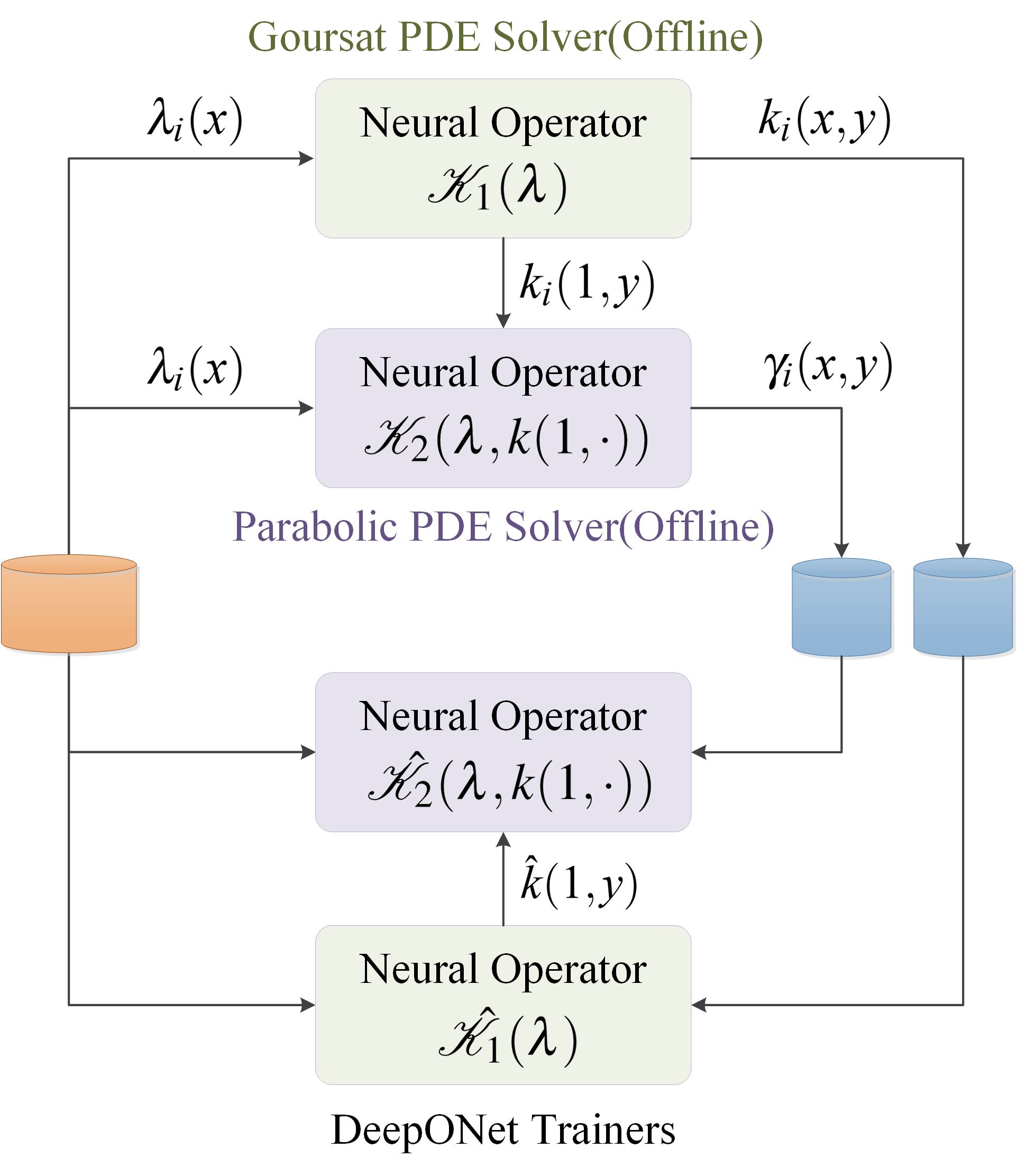}
\vspace{-0.21cm}
\caption{The process of learning the PDE backstepping design \textcolor{black}{in DeepONet involves two operators:
(1) for the first operator, $\lambda\mapsto k$, we compute multiple solutions of a kernel PDE \eqref{equ-kxx}--\eqref{equ-k0} in the Goursat form by substituting different functions $\lambda_i$.
(2) For the second operator described by the mapping $(\lambda,\ k(1,\cdot))\mapsto \gamma$, we solve the kernel PDE\eqref{equ-gammax}--\eqref{equ-gamma0}, which form likes a parabolic PDE, using the discretized form of the three-point difference method. This process is repeated multiple times with various functions $\lambda_i$ and their corresponding $k(1, y)$ values.} After this step, the neural operators $\hat{\mathcal{K}}_i,\ i=1,2$ are trained.} \label{Operator}
\end{figure}

\begin{theorem}
\label{thm-karniadakis-bkst}
{\bf\em [DeepONet approximation of kernels]}
Consider the neural operators defined in \eqref{equ-M-1} and \eqref{equ-M-2}, along with \eqref{equ-k-1}, \eqref{equ-k-2},  \eqref{equ-gamma-1}, \eqref{equ-kxx}--\eqref{equ-gamma0}. For all $B_\lambda,\ B_{\lambda'},\ B_{\lambda''}>0$ and $\epsilon>0$, there exist  neural operators  $\hat {{\mathcal{M}}}_1:C^1[0,1]\to \underline K\times C^1[0,1]\times \Upsilon $ and $\hat {\mathcal{M}}_2:C^1[0,1]\times C^1[0,1]\to  \underline \Gamma\times \Upsilon $ such that, %for all $(x,y)\in\Omega_2$,
\begin{align}
\nonumber&|{\mathcal M}_1(\lambda)(x,y)-\hat {\mathcal M}_1(\lambda)(x,y)|\\
&+|{\mathcal M}_2(\lambda, k(1,\cdot))(x,y)-\hat {\mathcal M}_2(\lambda, k(1,\cdot))(x,y)|<\epsilon,
\end{align}
holds for all Lipschitz $\lambda,\ k$ with the properties that $\rVert\lambda\rVert_\infty\leq B_\lambda$, $\rVert\lambda'\rVert_\infty\leq B_{\lambda'}$, $\rVert\lambda''\rVert_\infty\leq B_{\lambda''}$, namely, there exists neural operators $\hat {\mathcal{K}}_i,\ i=1,\ 2$ such that  
$ %\begin{align} &
\hat {\mathcal{K}}_1(\lambda)(x,0)\equiv0, 
%\\ &
\hat {\mathcal{K}}_2(\lambda,k(1,\cdot))(x,0)=\hat {\mathcal{K}}_2(\lambda,k(1,\cdot))(x,1)\equiv0
$,  %\end{align}
and 
\begin{align}\label{equ-epsilon}
\nonumber&|\tilde k(x,y)|+|2\frac{\mathrm d}{\mathrm dx}(\tilde k(x,x)|+|\tilde k_{xx}(x,y)-\tilde k_{yy}(x,y)-\lambda(y)\tilde k(x,y)|\\
\nonumber&+|\tilde \gamma(x,y)|+|\tilde \gamma_{x}(x,y)|+|\tilde \gamma_{x}(x,y)-D\tilde \gamma_{yy}(x,y)-D\lambda(y)\tilde \gamma(x,y)|\\
&|\tilde k_{yyy}(x,y)|+|\tilde \gamma_y(x-y,1)|+|\tilde \gamma_{xy}(x-y,1)|<\epsilon,
\end{align}
where 
\begin{align}
\nonumber\tilde k(x,y)=&k(x,y)-\hat k(x,y)\\
=&{\mathcal{K}}_1(\lambda)(x,y)-\hat {\mathcal{K}}_1(\lambda)(x,y), \\
\nonumber\tilde\gamma(x,y)=&\gamma(x,y)-\hat \gamma(x,y)\\
=&{\mathcal{K}}_2(\lambda, k(1,\cdot))(x,y)-\hat {\mathcal{K}}_2(\lambda,k(1,\cdot))(x,y).
\end{align}
\end{theorem}

\begin{proof}\em
    The continuity of the operators $\mathcal{M}_1$ and $\mathcal{M}_2$ follows from Lemmas \ref{lem1} and \ref{lem2}, respectively. The result is obtained by invoking \cite[Thm. 2.1]{lu2021advectionDeepONet}. 
\mbox{}\hfill$\blacksquare$
\end{proof}

Instead of approximating the multi-input map $(\mathcal{M}_1, \mathcal{M}_2)$ in Theorem \ref{thm-karniadakis-bkst}, we can obtain the remaining results of the paper also by considering the composition operator $\mathcal{M}= (\mathcal{M}_1, \mathcal{M}_2(\mathcal{M}_1)): \lambda \mapsto (k,K_1, K_2, \gamma, \Gamma)$, whose sole input function is $\lambda$. This appears more elegant but, in actual learning of the neural version of the operator, $\hat{\mathcal{M}}$, 
the training loss is bigger than for the neural operator $(\hat{\mathcal{M}}_1, \hat{\mathcal{M}}_2)$, especially for smaller datasets, so we stick with the approach %as stated 
in  Theorem \ref{thm-karniadakis-bkst}. 

In  Section \ref{Stab-Deep}, %we establish that 
Theorem \ref{thm-karniadakis-bkst}
%the tailored  Universal Approximation Operator Theorem stated above (see Reference \cite{lu2021learning}), 
enables us to prove exponential stabilization of \eqref{equ-u-bis}--\eqref{equ-ucas0} with the control law  \eqref{equ-U} when the gain kernels are approximated via  DeepOnet using a collection of input-output data generated with 
%by considering 
{the spatially varying reactivity $\lambda(x)$}. For the control loop  depicted in Figure \ref{Learning-process}, we  prove that the gain learned offline as a  DeepOnet enforces closed-loop  stability with a quantifiable exponential decay rate.
\begin{figure}[t]
\centering
\includegraphics[width=0.485\textwidth]{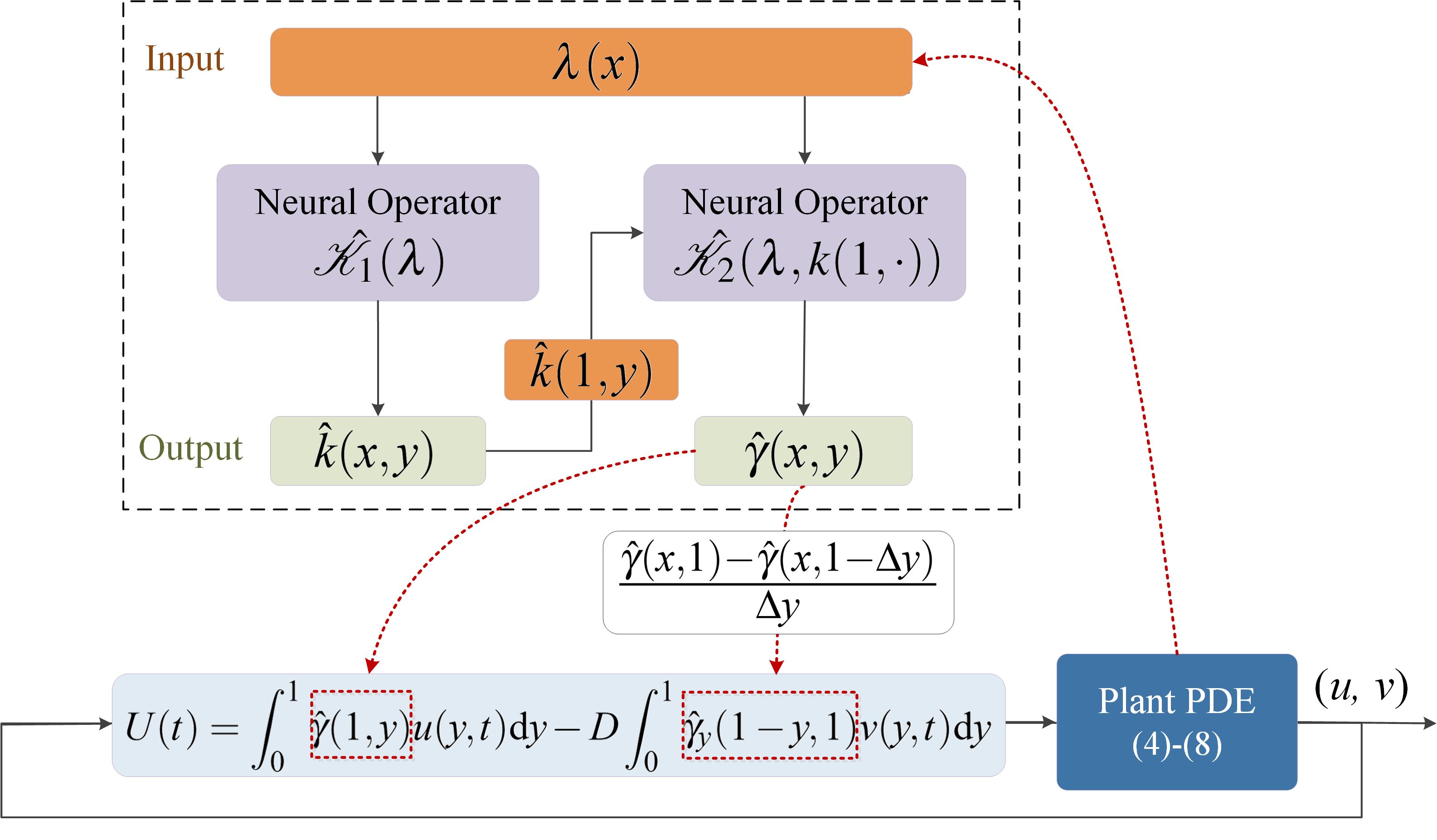}
\vspace{-0.21cm}
\caption{In the PDE backstepping control, the gain kernels are defined as $\gamma(x,y)$ and $\gamma_y(1-y,1)$. The first gain kernel is generated by the DeepONet $\hat{\mathcal K}_i,\ i=1,2$, is implemented by incorporating the learned neural operators into the control system, while the second gain kernel is obtained by numerically differentiating the first gain kernel directly.} \label{Learning-process}
\end{figure}

\section{Stabilization under DeepONet Gain Feedback}\label{Stab-Deep}

The backstepping transformations \eqref{equ-tranwu}, \eqref{equ-tranzv}  fed by the approximate gain kernel $\hat k(x,y)$ are defined as
\begin{align}\label{approx-trans-w}
\hat w(x,t)=&u(x,t)-\int_0^x\hat k(x,y)u(y,t)\mathrm dy,\\
\nonumber\hat z(x,t)=&v(x,t)-\int_0^1\hat \gamma(x,y)u(y,t)\mathrm dy\\
&-D\int_0^x\hat q(x-y)v(y,t)\mathrm dy,\label{approx-trans-z}
\end{align}
where $\hat k=\hat {\mathcal{K}}_1(\lambda)$, $\hat \gamma=\hat{\mathcal{K}}_2(\lambda, k(1,\cdot))$ and $\hat q(x)=-\hat \gamma_y(x,1)$ together with the approximate   control law
\begin{align}
U(t)=\int_0^1\hat \gamma(1,y)u(y,t)\mathrm dy-D\int_0^1\hat \gamma_y(1-y,1)v(y,t)\mathrm dy.\label{equ-hat-U}
\end{align}
By differentiation, equations  \eqref{approx-trans-w}, \eqref{approx-trans-z} and \eqref{equ-kxx}--\eqref{equ-q0} generate the   perturbed target system given below  
\begin{align}
\hat w_t(x,t)&=\hat w_{xx}(x,t)+\delta_1(x)u(x,t)+\int_0^x\delta_2(x,y)u(y,t)\mathrm dy,\label{equ-w0-hat}\\
\hat w(0,t)&=0,\\
\hat w(1,t)&=\hat z(0,t),\\
D\hat z_t(x,t)&=\hat z_x(x,t)+\int_0^1\delta_3(x,y)u(y,t)\mathrm dy,\\
\hat z(1,t)&=0,\label{equ-z1}
\end{align}
where $\delta_i,\ i=1,2,3$ are defined as 
\begin{align}
\delta_1(x)= &2\frac{\mathrm d}{\mathrm dx}\hat k(x,x)+\lambda(x)\nonumber\\
=&-2\frac{\mathrm d}{\mathrm dx}\tilde k(x,x),\\
\delta_2(x,y)=&\hat k_{xx}(x,y)-\hat k_{yy}(x,y)-\lambda(y)\hat k(x,y)\nonumber\\
=&-\tilde k_{xx}(x,y)+\tilde k_{yy}(x,y)+\lambda(y)\tilde k(x,y),\\
\delta_3(x,y)=&\hat \gamma_{x}(x,y)-D\hat \gamma_{yy}(x,y)-D\lambda(y)\hat \gamma(x,y)\nonumber\\
=&-\tilde\gamma_{x}(x,y)+D\tilde\gamma_{yy}(x,y)+D\lambda(y)\tilde \gamma(x,y).
\end{align}
Note that from \eqref{equ-epsilon},  the following inequalities hold
\begin{align}
&\rVert\delta_i\rVert_\infty\leq \epsilon,\quad i=1,2,3,\\
&\rVert\delta_{3x}\rVert_\infty\leq \epsilon.
\end{align}

Our first result for the closed-loop system is its exponential stability in the backstepping-transformed variables under the DeepONet-approximated kernels. 

\begin{propo}\label{prop-0}
{\bf\em [Lyapunov analysis for DeepONet-perturbed target system]}
Consider the target system \eqref{equ-w0-hat}--\eqref{equ-z1}. For each $D>0$, there exists $\varepsilon^*>0$ such that for all $\varepsilon\in (0,\varepsilon^*)$ there exist $c_1,\ c_2>0$ such that the following holds,
\begin{align}
\label{eq-Psi-exp-bound}
\Psi_1(t)\leq \Psi_1(0)c_2 {\rm e}^{-c_1 t}, \quad \forall \geq 0,
\end{align}
where
\begin{align}
\Psi_1(t)=\rVert\hat w(t)\rVert^2+\rVert\hat z(t)\rVert^2+\rVert\hat z_x(t)\rVert^2.\label{equ-Psi}
\end{align}
\end{propo}

The proof of Proposition \ref{prop-0} is given in Appendix \ref{lemma-3-proof}.

To translate the stability of the target system into that of the original closed-loop system,  {we consider transformations \eqref{approx-trans-w} and \eqref{approx-trans-z}, along with their inverse transformations, 
%\eqref{approx-inverse0} and 
\begin{align}
u(x,t)=&\hat w(x,t)+\int_0^x\hat l(x,y)\hat w(y,t)\mathrm dy,\label{equ-inv-trans-u}\\
v(x,t)=&\hat z(x,t)+\int_0^1\hat \eta(x,y)\hat w(y,t)\mathrm dy\nonumber\\
&+D\int_0^x\hat p(x-y)\hat z(y,t)\mathrm dy,\label{equ-inv-trans-v}
\end{align}
and provide the following proposition.
%we state Proposition \ref{propo1}.

\begin{propo}\label{propo1}
{\bf\em [norm equivalence with DeeepONet kernels]}
The following estimates hold  between the plant  \eqref{equ-u-bis}--\eqref{equ-ucas0}  and the target system \eqref{equ-w0-hat}--\eqref{equ-z1}, 
\begin{align}\label{S1S2}
S_1\Psi_1(t)\leq \Phi(t) \leq S_2\Psi_1(t)
\end{align}
where 
\begin{align}
\Phi(t)=\rVert u(t)\rVert^2+\rVert v(t)\rVert^2+\rVert v_x(t)\rVert^2,
\end{align}
and  the positive constants
\begin{align}
\nonumber S_1=&9+2\int_0^1\int_0^x\hat k(x,y)^2\mathrm dx\mathrm dy+3\int_0^1\int_0^1\hat \gamma(x,y)^2\mathrm dx\mathrm dy\\
\nonumber&+4\int_0^1\int_0^1\hat \gamma_x(x,y)^2\mathrm dx\mathrm dy+3D^{2}\int_0^1\int_0^x\hat q(x-y)^2\mathrm dx\mathrm dy\\
&+4D^{2}\hat q(0)^2+4D^{2}\int_0^1\int_0^x\hat q'(x-y)^2\mathrm dx\mathrm dy,\label{equ-S-1-0}\\
\nonumber S_2=&9+2\int_0^1\int_0^x\hat l(x,y)^2\mathrm dx\mathrm dy+3\int_0^1\int_0^1\hat \eta(x,y)^2\mathrm dx\mathrm dy\\
\nonumber&+4\int_0^1\int_0^1\hat \eta_x(x,y)^2\mathrm dx\mathrm dy+3D^{2}\int_0^1\int_0^x\hat p(x-y)^2\mathrm dx\mathrm dy\\
&+4D^{2}\hat p(0)^2+4D^{2}\int_0^1\int_0^x\hat p'(x-y)^2\mathrm dx\mathrm dy\label{equ-S-2-0}
\end{align}
%where $S_i,\ i=1,2$ 
are bounded.
%based on the kernel $k(x,y)$, $k_y(x,y)$, $k_{yy}(x,y)$, and $k_{yyy}(x,y)$.
\end{propo}

%wang2021adaptive
\begin{proof}\em
In Appendix \ref{appendix-B}. Similar to  proof of \cite[Prop. 1]{Wang2021adaptive}.
\mbox{}\hfill$\blacksquare$
\end{proof}

Based on the norm-equivalence in Proposition \ref{propo1}, the  main result in the next theorem immediately follows from Proposition \ref{prop-0}.

\begin{theorem}\label{theo2}
{\bf\em [main result---stabilization with DeepONet]}
Consider the system \eqref{equ-u-bis}--\eqref{equ-ucas0} or, equivalently,  system \eqref{equ-u}--\eqref{equ-u-bud-U},  with any $\lambda\in {C}^1([0,1])$ whose derivatives $\lambda'$ and $\lambda''$ are Lipschitz and which satisfy $\rVert\lambda\rVert_\infty\leq B_\lambda$, $\rVert\lambda'\rVert_\infty\leq B_{\lambda'}$, $\rVert\lambda''\rVert_\infty\leq B_{\lambda''}$, and $B_\lambda,\ B_{\lambda'},\ B_{\lambda''}>0$. There exists a sufficiently small $\epsilon^*(B_\lambda,\ B_{\lambda'},\ B_{\lambda''})>0$, such that the feedback controller \begin{align}\label{Learned-control}
U(t)=\int_0^1\hat \gamma(1,y)u(y,t)\mathrm dy-D\int_0^1\hat \gamma_y(1-y,1)v(y,t)\mathrm dy,
\end{align} 
with two neural operators $\hat {\mathcal{M}}_i,\ i=1,2$ of approximation accuracy $\epsilon\in(0,\epsilon^*)$, in relation to the exact backstepping kernels $k(x,y)$ and $\gamma(x,y)$, respectively, ensure that the closed-loop system satisfies the following exponential stability bound,
%\eqref{stabilitybound}.
\begin{align}\label{stabilitybound}
\Phi(t)\leq \Phi(0)\frac{S_2}{S_1}c_2 {\rm e}^{-c_1 t}\,,
\quad\forall t\geq 0\,.
\end{align}
\end{theorem}

\section{Simulation of the Stability results under the NO approximated gain kernels}\label{Simul-kernel}

Our simulation is for a reaction-diffusion PDE with a boundary input delay $D = 2$ and a spatially varying reactivity coefficient $\lambda(x) = 10.2+$ $2\cos(\Gamma \cos^{-1}(x))$ parameterized by $\Gamma=[5, 10]$, which renders the PDE unstable for all $\Gamma$s in that range.
%Under a constant initial condition  $u_0(x) =10$,  t
The closed-loop system with the nominal, delay-uncompensated backstepping controller is unstable as shown in Figure \ref{closed-loop-uncom} {with the spatial step $\Delta x=0.005$ and  the temporal step   $\Delta t=0.0001$}. 
Additionally,  the plots of the reaction terms depicted in Figure \ref{lambda-x} show that the intensity of the oscillations of $\lambda$ is an increasing function of $\Gamma$, leading to a lower rate of instability, as indicated on the right side of Figure \ref{lambda-x}.
\begin{figure}
\centering
\includegraphics[width=0.485\textwidth]{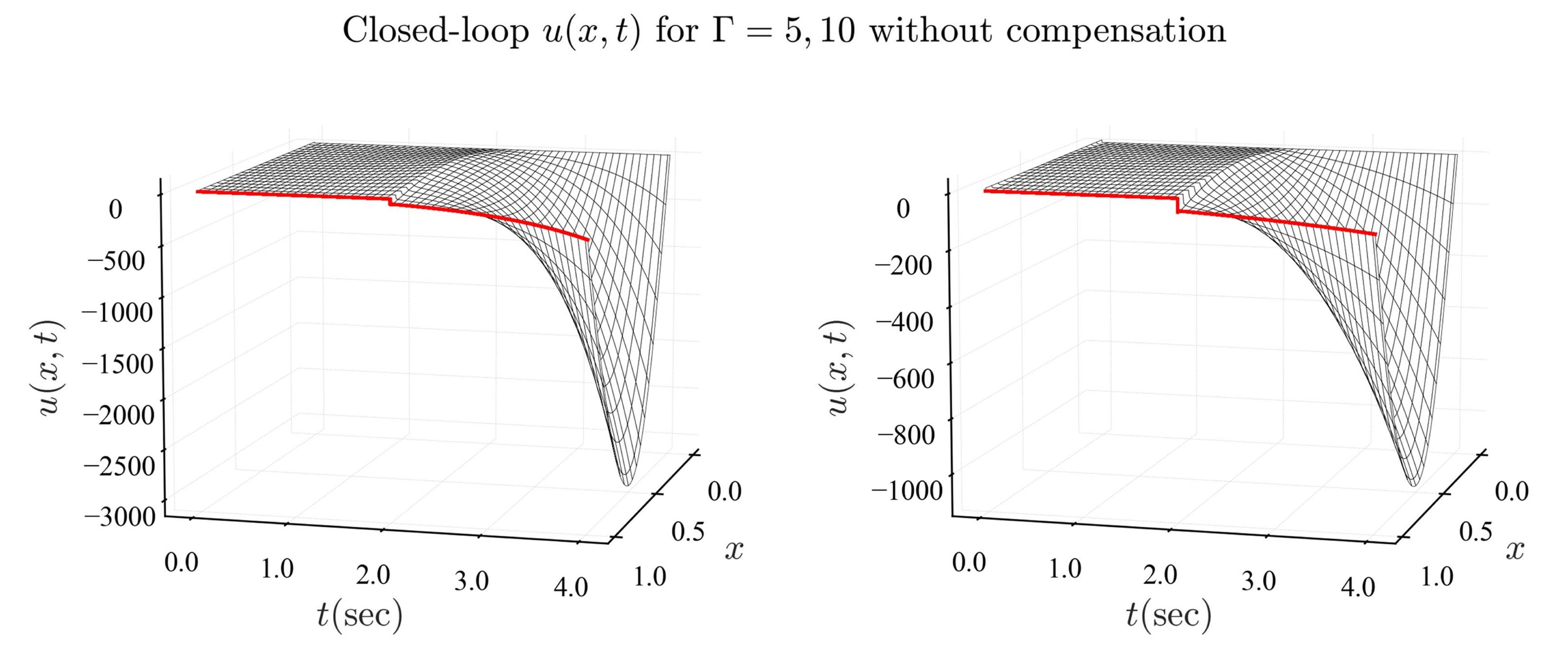}
\vspace{-0.21cm}
\caption{{Instability of the closed-loop system with  reaction coefficients defined as $\lambda(x)=10.2+2\cos(\Gamma\cos(x)^{-1})$, $\Gamma=5,\ 10$, respectively, under the control law of the delay-free plant \cite{krstic2023neural} for a constant initial condition  $u_0(x) =10$.}} \label{closed-loop-uncom}
\end{figure}
\begin{figure}
\centering
\includegraphics[width=0.485\textwidth]{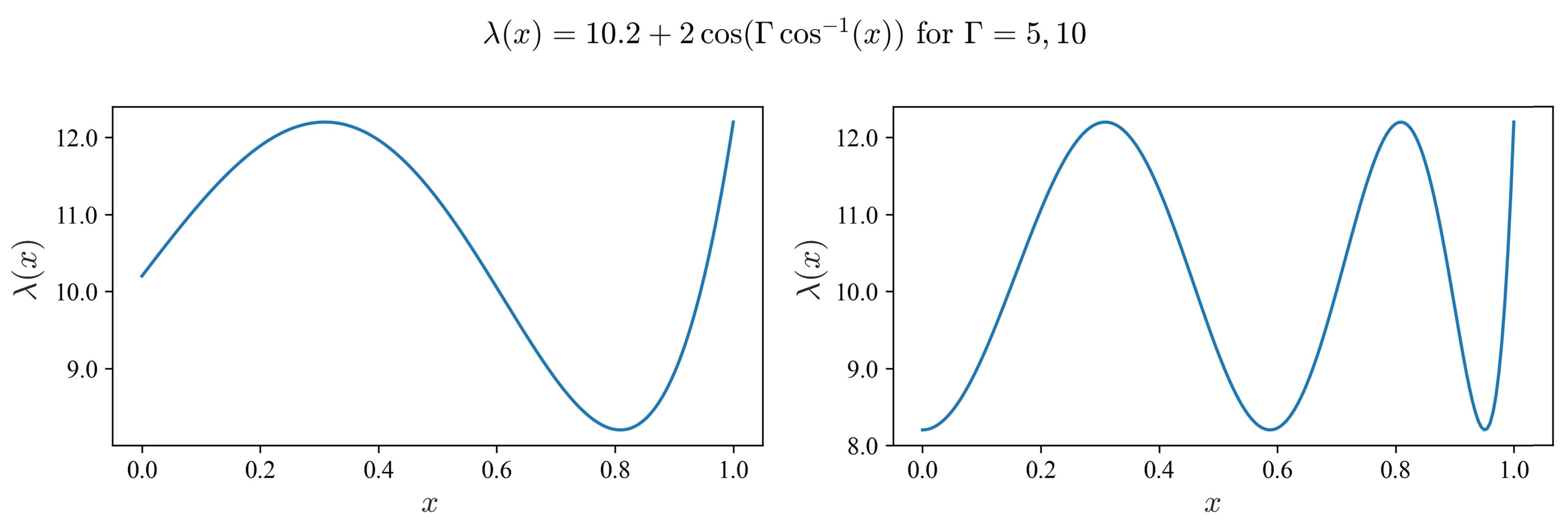}
\vspace{-0.21cm}
\caption{The reaction coefficients $\lambda(x)=10.2+2\cos(\gamma\cos(x)^{-1})$ with $\gamma=5,\  10$, respectively.} \label{lambda-x}
\end{figure}
As opposed to \cite{krstic2023neural}, we consider learning the two kernels $k(x,y)$ and $\gamma(x,y)$, which can be achieved in two steps by defining the mappings $\mathcal{K}_1:\lambda \mapsto k$ and  $\mathcal{K}_2: (\lambda,\  k(1,\cdot))\mapsto \gamma$. The first learning step has been performed in \cite{krstic2023neural} but, for the second step, a DeepONet with two inputs and one output is required. The training of our neural network on the approximate kernel $\gamma(x,y)$ is accomplished by passing the training of two datasets of different $\lambda(x)$ and $k(1,y)$ with a given range of the parameter $\Gamma$, namely, $\Gamma \in [5,10]$, uniformly distributed through the network.

\begin{table}[t]
\centering
\resizebox{\columnwidth}{!}{
\begin{tabular}{|c|c|c|c|c|c|}
  \hline
 \begin{tabular}[c]{@{}c@{}}Data Sets \\ (train/test)\end{tabular} & \begin{tabular}[c]{@{}c@{}}Training\\ Time\\(sec)\end{tabular} & \begin{tabular}[c]{@{}c@{}}Average\\ Relative\\ Error \\Training \\Data\end{tabular} & \begin{tabular}[c]{@{}c@{}}Average\\ Relative\\ Error \\Testing \\Data\end{tabular} & \begin{tabular}[c]{@{}c@{}}Average\\
Calculation\\Time CPU\\ (sec)\end{tabular} &\begin{tabular}[c]{@{}c@{}}Average\\
Calculation\\Time GPU\\ (sec)\end{tabular} \\
  \hline
   \begin{tabular}[c]{@{}c@{}}500\\(450/50)\end{tabular} & \textcolor{blue}{325.76} & $1.02\times 10^{-2}$ & $1.05\times 10^{-2}$\ & \textcolor{blue}{$1.09$} & $6.75\times 10^{-3}$\\
  \hline
   \begin{tabular}[c]{@{}c@{}}1000\\(900/100)\end{tabular} & {609.84} &  $7.72\times 10^{-4}$ & $7.86\times 10^{-4}$ & $1.74$ &$3.87\times 10^{-3}$\\
  \hline
     \begin{tabular}[c]{@{}c@{}}2000\\(1800/200)\end{tabular} & 1201.61 & \textcolor{blue}{$2.16\times 10^{-4}$} & \textcolor{blue}{$2.44\times 10^{-4}$} & {$2.34$} &\textcolor{blue}{$2.03\times 10^{-3}$}\\
\hline
\end{tabular}}
\caption{The relationship between different sizes of datasets and the accuracy of NO approximated $k$ under the parameter size 76169285. The best values are highlighted in blue.}
\label{table-k}
\end{table}

\begin{table}[t]
\centering
\resizebox{\columnwidth}{!}{
\begin{tabular}{|c|c|c|c|c|c|}
  \hline
 \begin{tabular}[c]{@{}c@{}}Data Sets \\ (train/test)\end{tabular} & \begin{tabular}[c]{@{}c@{}}Training\\ Time\\(sec)\end{tabular} & \begin{tabular}[c]{@{}c@{}}Average\\ Relative\\ Error \\Training \\Data\end{tabular} & \begin{tabular}[c]{@{}c@{}}Average\\ Relative\\ Error \\Testing \\Data\end{tabular} & \begin{tabular}[c]{@{}c@{}}Average\\
Calculation\\Time CPU\\ (sec)\end{tabular} &\begin{tabular}[c]{@{}c@{}}Average\\
Calculation\\Time GPU\\ (sec)\end{tabular} \\
  \hline
   \begin{tabular}[c]{@{}c@{}}500\\(450/50)\end{tabular} & \textcolor{blue}{323.52} & $9.94\times 10^{-3}$  & $1.09\times 10^{-2}$ & \textcolor{blue}{$1.41$} & {$7.97\times 10^{-3}$}\\
  \hline
   \begin{tabular}[c]{@{}c@{}}1000\\(900/100)\end{tabular} & {628.61} &  $2.78\times 10^{-3}$ & $2.80\times 10^{-3}$ & $1.50$ &$3.82\times 10^{-3}$\\
  \hline
     \begin{tabular}[c]{@{}c@{}}2000\\(1800/200)\end{tabular} & 1235.74 & \textcolor{blue}{$9.92\times 10^{-4}$} & \textcolor{blue}{$1.01\times 10^{-3}$} & {$2.34$} &\textcolor{blue}{$2.03\times 10^{-3}$}\\
\hline
\end{tabular}}
\caption{The relationship between different sizes of datasets and the accuracy of NO approximated $\gamma$ under the parameter size 76169685. The most accurate approximations are highlighted in blue.}
\label{table-gamma}
\end{table}
\begin{figure}[t]
\centering
\includegraphics[width=0.485\textwidth]{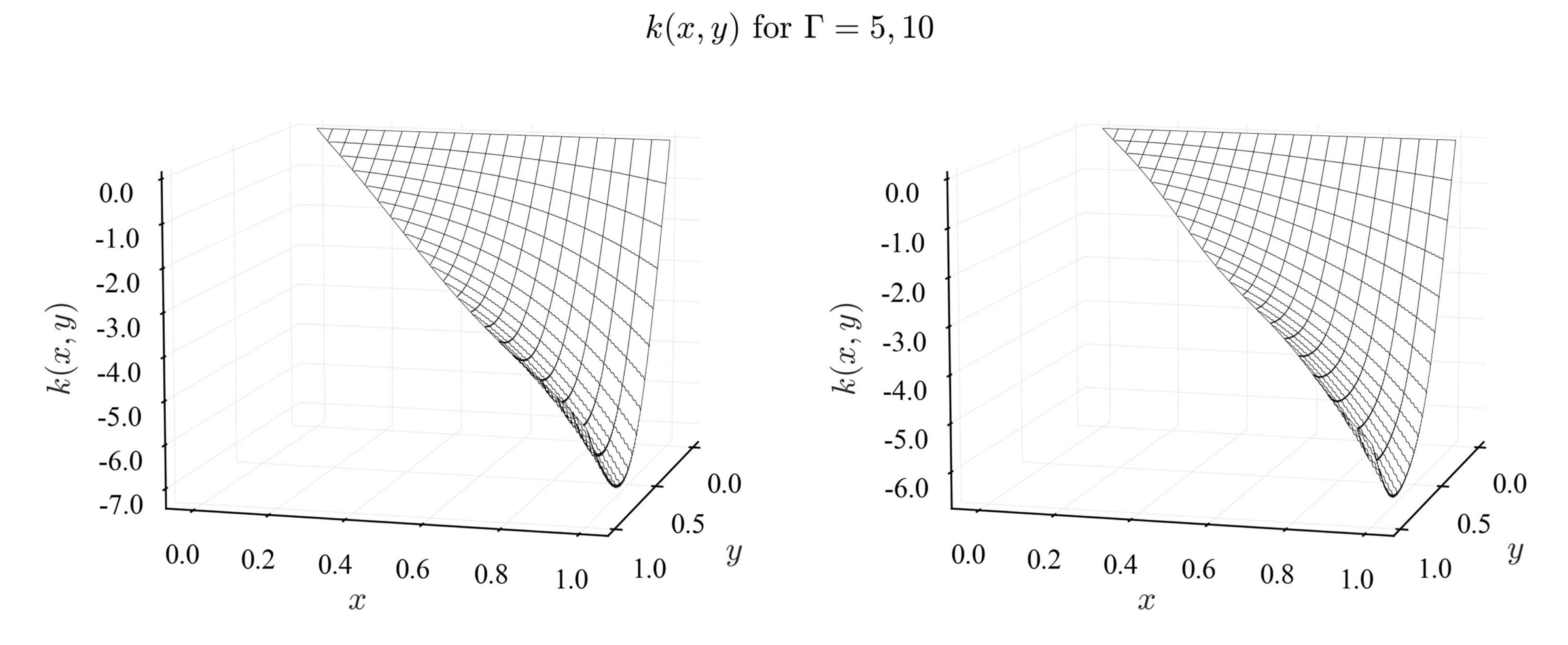}
\includegraphics[width=0.485\textwidth]{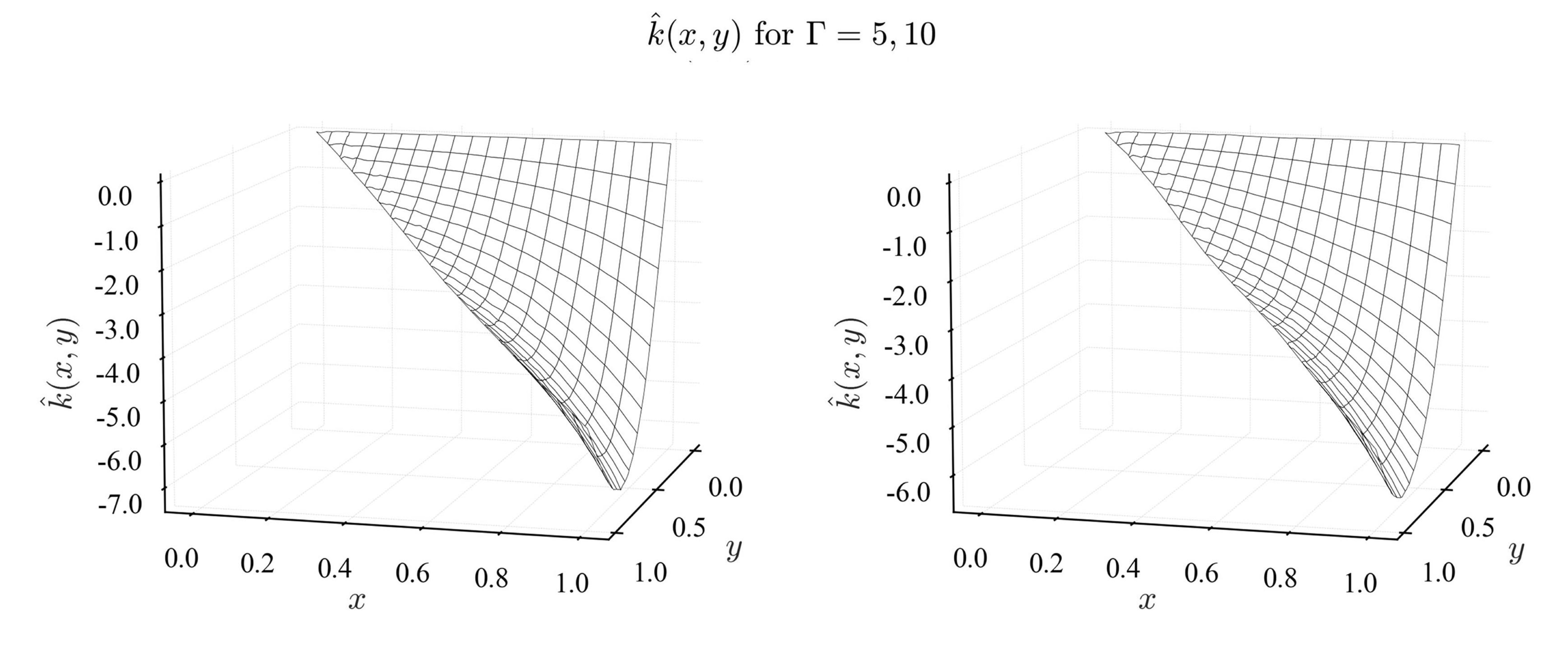}
\includegraphics[width=0.485\textwidth]{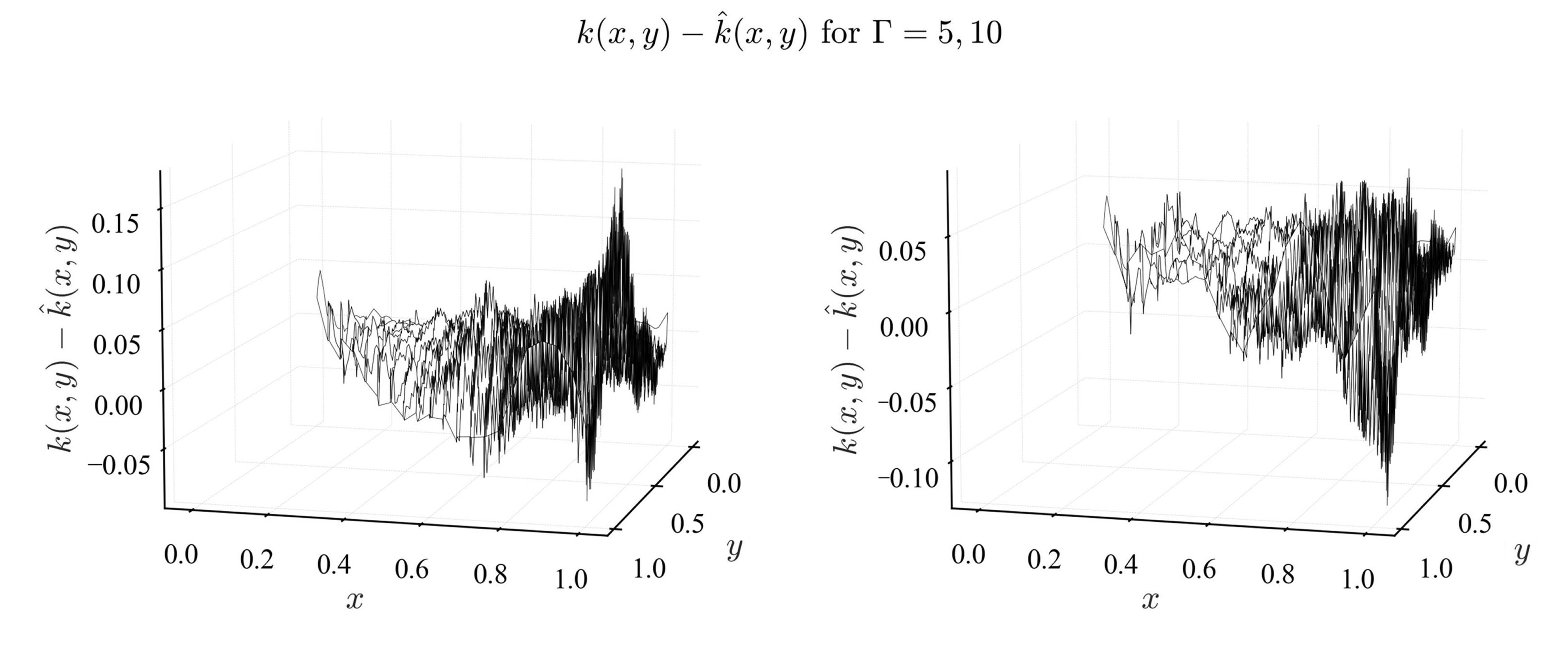}
\caption{The kernels $k(x,y)$, $\hat k(x,y)$ and $k(x,y)-\hat k(x,y)$.} \label{figure-k}
\end{figure}

The DeepONet is efficiently used without an alteration of the grid, by iterating  $\lambda(x)$  along the $y$-axis to generate a 2D input for  $\mathcal{K}_1$ network. For similar reasons,   $\lambda(x),\ k(1,y)$ along the $y$-axis and $x$-axis, respectively, are to create two 2D input for $\mathcal{K}_2$ network. Our approach leverages the 2D structure by employing a CNN for the DeepONet branch network. {The relationship between various dataset sizes and the accuracy of the approximate values of $k$ and $\gamma$ is presented in Tables \ref{table-k} and \ref{table-gamma}, respectively. {The stepsizes for $x$ and $y$ are both $0.005$. All experiments are run on an Nvidia Titan Xp (the CPU is Intel I7).} Based on a dataset of size 2000, the model that achieves the utmost precision in classifying data points is obtained. Figures \ref{figure-k} and \ref{figure-gamma} illustrate the analytical and learned DeepONet kernels $k$ and $\gamma$, respectively, for the two $\Gamma$ values depicted in Figure \ref{lambda-x}}. During the training process, see Figure \ref{Train-loss}, the relative $L^2$ errors for kernels $k(x,y)$ and $\gamma(x,y)$ are measured to be $2.16\times 10^{-4}$ and $9.92\times 10^{-4}$, while the testing error amounts to $2.44\times 10^{-4}$ and $1.01\times 10^{-3}$,
respectively. By employing the learned neural operator, we achieve notable speedups, approximately on the order of $10^3$, compared to an efficient finite-difference implementation. To compute the approximate value of the controller signal given by \eqref{equ-hat-U},  the value of $\hat \gamma_y(1-y,1)$ is needed, which can be obtained by using the estimate of $\hat\gamma_y(x,1)=\frac{\hat\gamma(x,1)-\hat\gamma(x,1-\Delta y)}{\Delta y}$. The simulations confirm closed-loop stability under a delay-compensated boundary control with kernel gains approximation as shown in Fig. \ref{figure-u}.
%, which illustrates the feasibility of the proposed method.

\begin{figure}[t]
\centering
\includegraphics[width=0.485\textwidth]{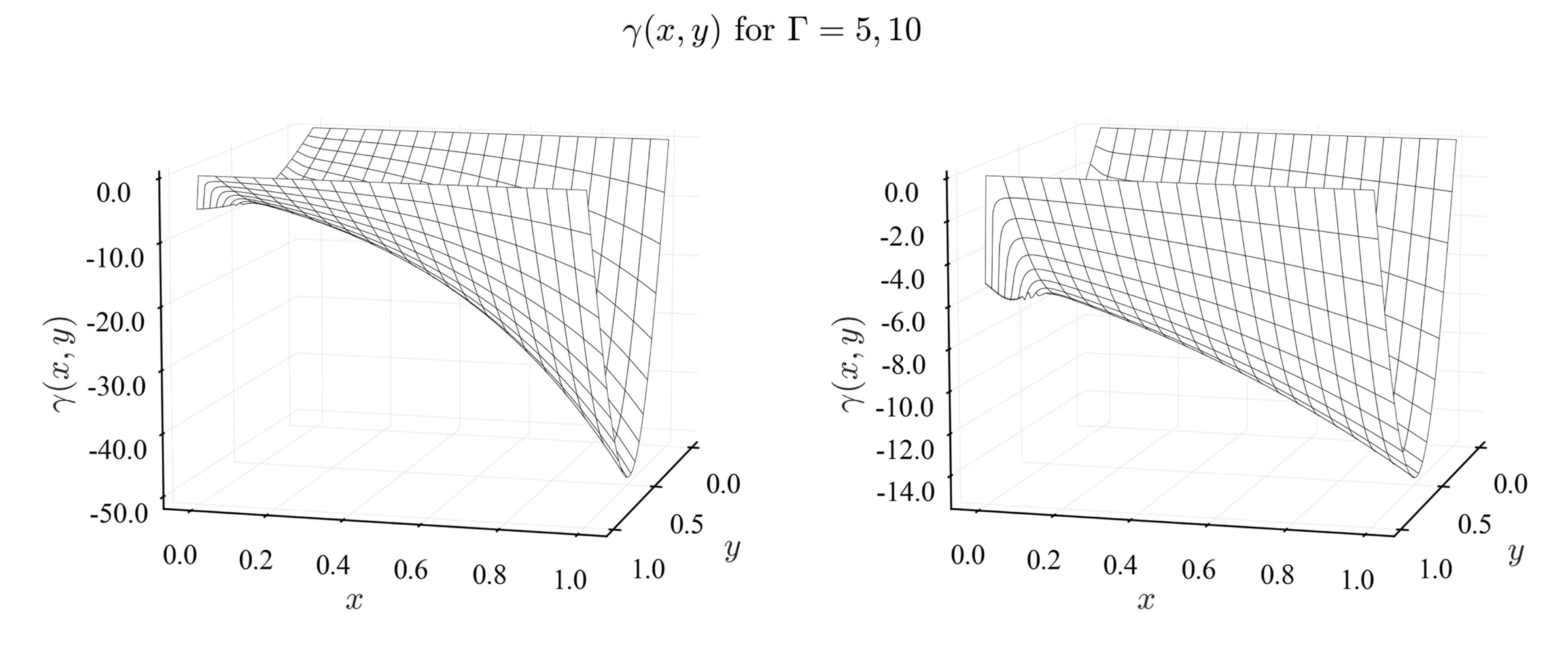}
\includegraphics[width=0.485\textwidth]{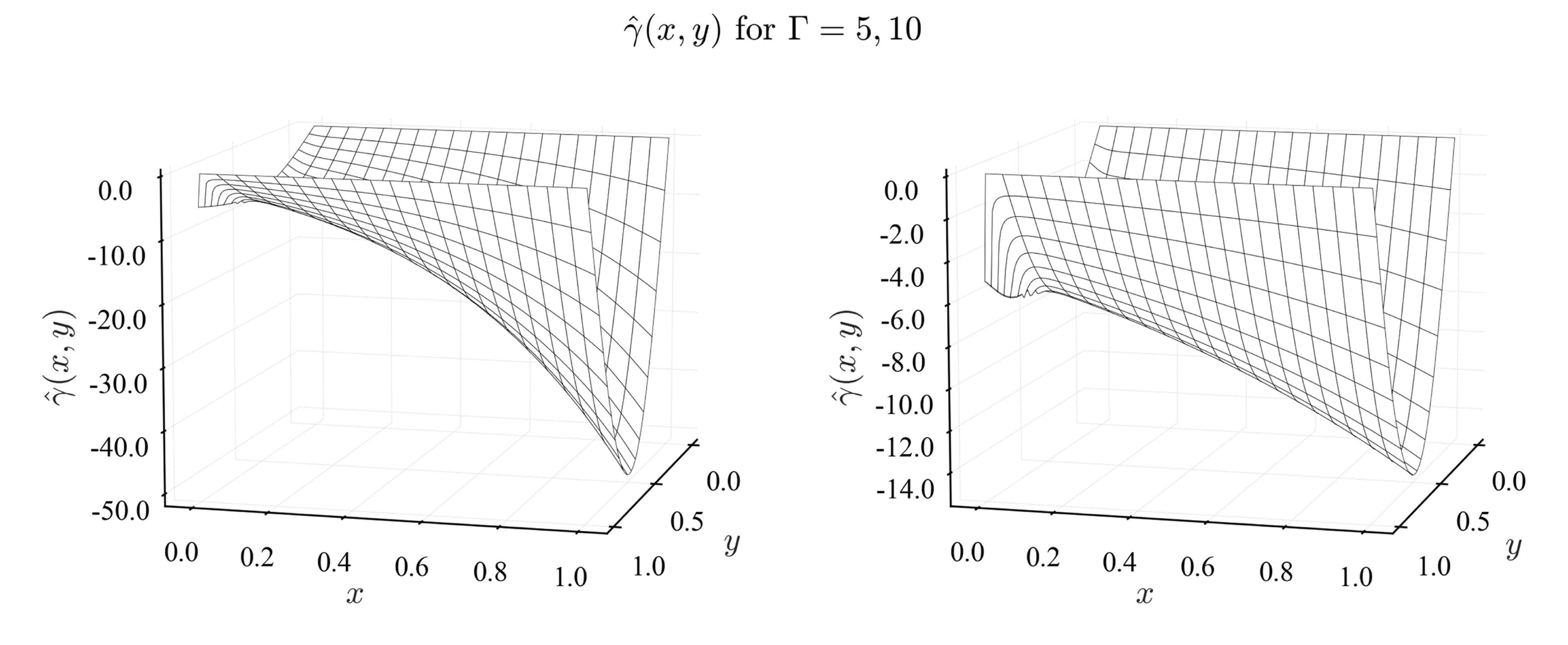}
\includegraphics[width=0.485\textwidth]{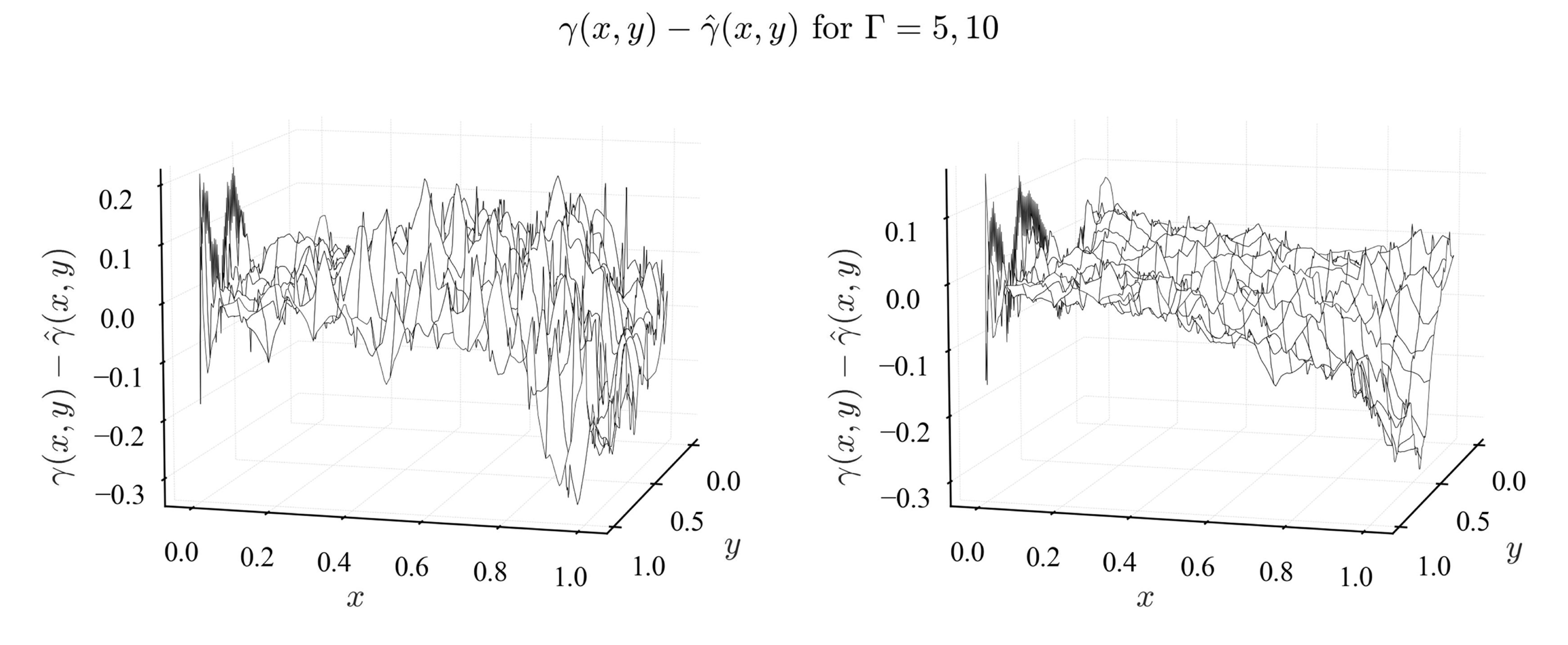}
\caption{The kernels $\gamma(x,y)$, $\hat \gamma(x,y)$ and $\gamma(x,y)-\hat \gamma(x,y)$} 
\label{figure-gamma}
\end{figure}

\begin{figure}[t]
\centering
\includegraphics[width=0.24\textwidth]{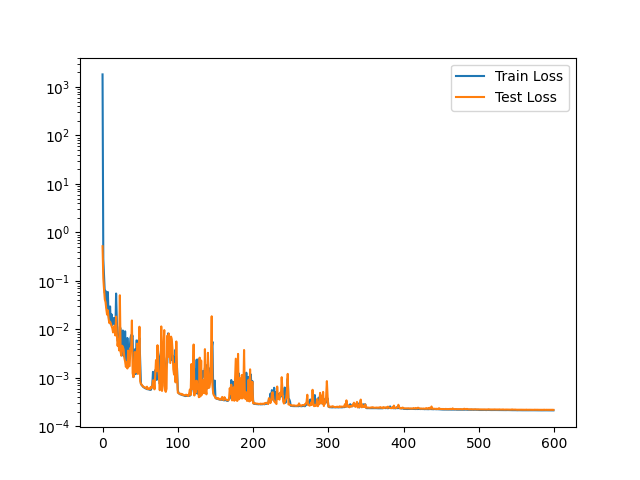}
\includegraphics[width=0.24\textwidth]{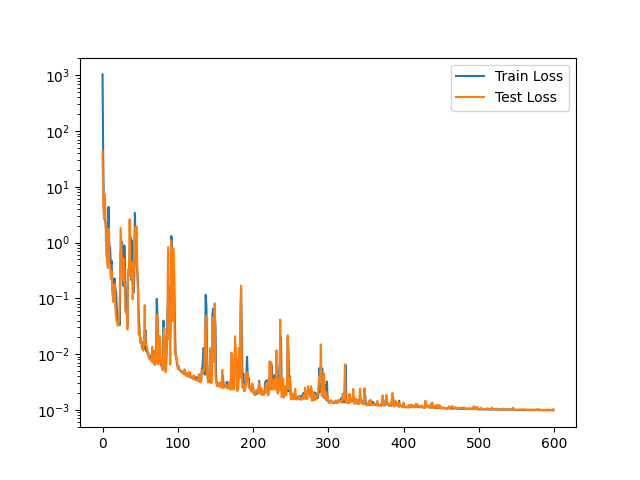}
\vspace{-0.21cm}
\caption{The train and test loss for $\lambda\mapsto k$ and $(\lambda,\ k(1,\cdot))\mapsto \gamma$.} \label{Train-loss}
\end{figure}

\begin{figure}[t]
\centering
\includegraphics[width=0.485\textwidth]{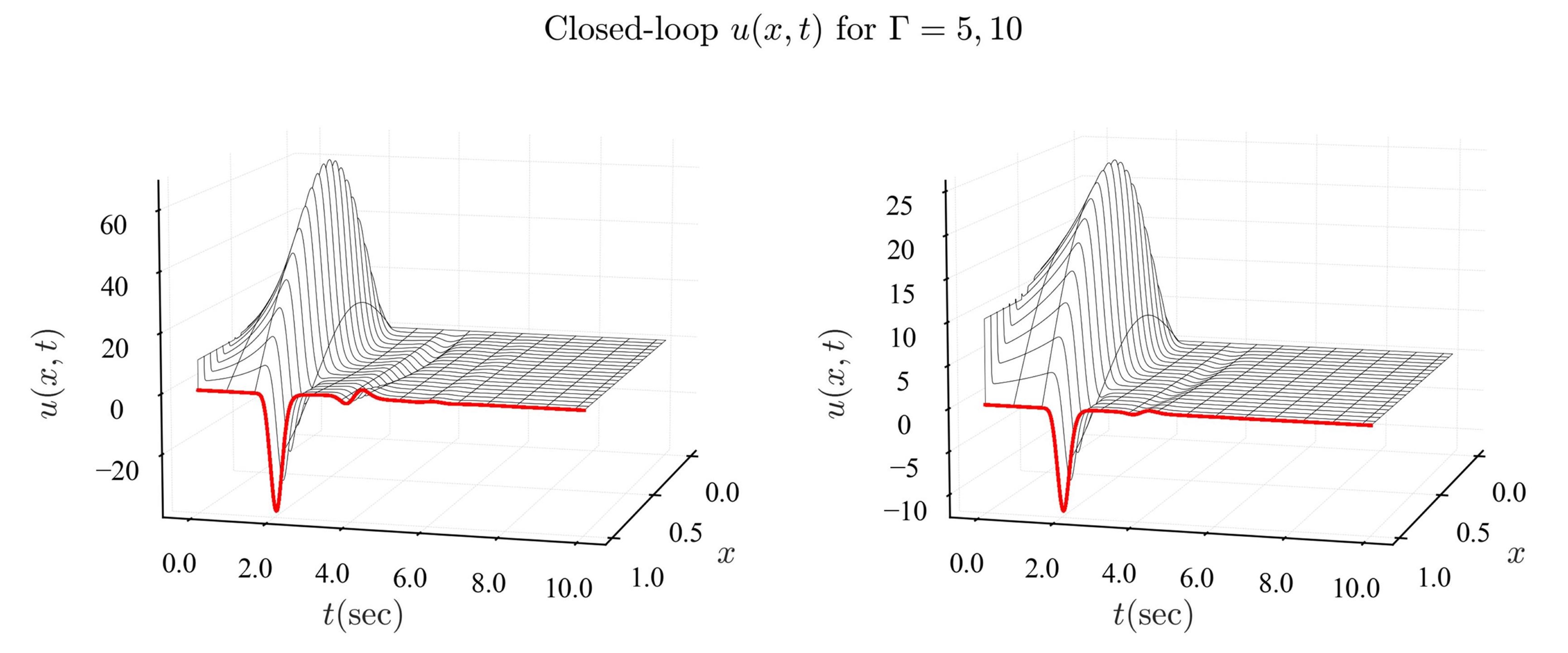}
\includegraphics[width=0.485\textwidth]{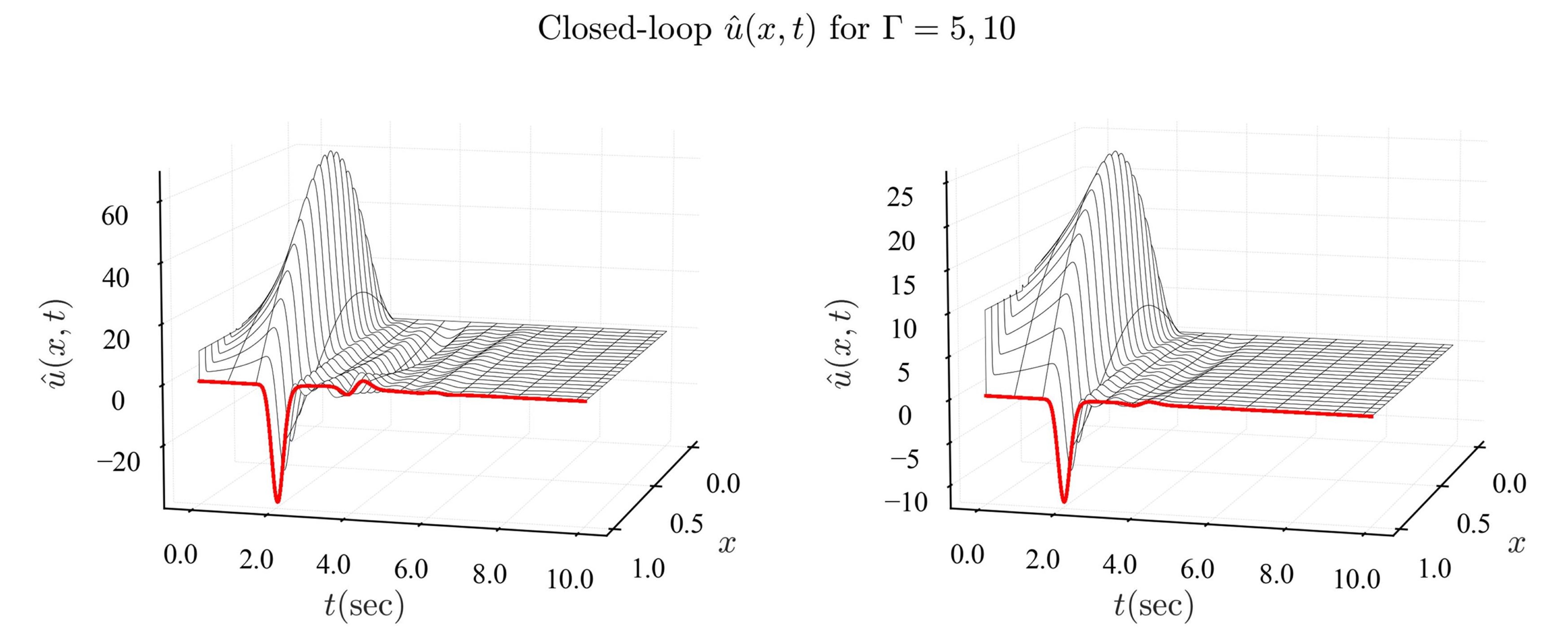}
\vspace{0.21cm}
\includegraphics[width=0.485\textwidth]{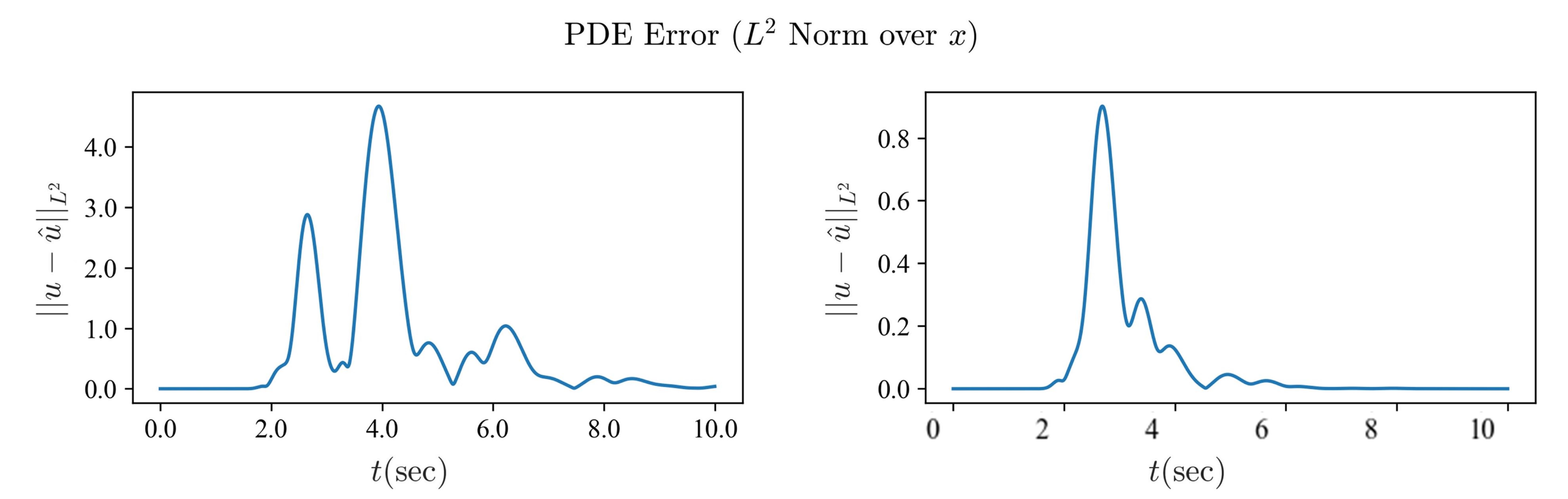}
\caption{For the two respective values of $\lambda(x)$, the top row shows closed-loop solutions with kernels $k(x,y)$ and $\gamma(x,y)$, the middle row shows closed-loop solutions with the learned kernels $k(x,y)$ and $\gamma(x,y)$, and the bottom row shows the closed-loop PDE error between applying the original kernels $k(x,y),\ \gamma(x,y)$ and the learned kernels $k(x,y)$ and $\gamma(x,y)$.}\label{figure-u}
\end{figure}

\section{Concluding Remarks}\label{conclude}

The papers \cite{bhan2023neural}, \cite{krstic2023neural}, and \cite{qi2023neural} introduced DeepONet approximations of PDE backstepping designs with single Goursat-form kernel PDEs for, respectively, hyperbolic, parabolic, and hyperbolic PDEs with input delays. By considering a backstepping design for a reaction-diffusion PDE with input delay in this paper, we demonstrate how to tackle problems where more than one kernel PDE, and belonging to different PDE classes, need to be approximated with a DeepONet, and their stabilizing capability proven using the universal approximation theorem for nonlinear operators \cite{lu2021advectionDeepONet}. 

The structure of the nonlinear operator being approximated is  interesting and instructive. First, the reactivity function generates one backstepping kernel, which is a solution to a second-order parabolic Goursat PDE. Then, the solution of that Goursat PDE serves as an initial condition to another parabolic (reaction-diffusion) PDE, to generate a second (predictor) kernel. Each of the two PDEs, from distinct classes, gives rise to a separate nonlinear operator and, taken as a composition of operators, they generate a single nonlinear operator, from the reactivity function to the control gain functions. We approximate this overall nonlinear operator to generate the control input. 

Thus, the extension of DeepONet-based PDE backstepping to compositions of kernel PDE-based nonlinear operators is the paper's main technical contribution. The paper's control contributions are the proofs of stabilization under the approximated kernels.

\appendices 

\section{Proof of Lemma \ref{lem2}}\label{appendix-A}

Let's define the $L^2$ norm of function $f(x,t)\in L^2[0,1]$ as follows:
\begin{align}
\rVert f(t)\rVert_{L^2}^2=\int_0^1|f(x,t)|^2\mathrm dx.
\end{align}

To prove  Lemma \ref{lem2}, we start from the following lemma, which derives explicit bounds on the derivatives of $k(x,y)$, and then, after this lemma, turn our attention to $\gamma(x,y)$ from  Lemma \ref{lem2}. 

\begin{lemm}\label{lemma-k}
{\bf\em [bounds on Goursat kernel's derivatives]}
{For the function $k(x,y)$  satisfying the PDE system \eqref{equ-kxx}--\eqref{equ-k0},
%G(\xi,\eta)=k(x,y)$,  where $\xi = x-y, \eta = x-y$, and $k(x,y)$ satisfies $|k(x,y)|\leq \bar \lambda {\rm e}^{2\bar \lambda x}$ is established in \eqref{equ-k-bouded}, 
the following holds:
\begin{align}
|k_{y}(x,y)|^2\leq &\frac{5\bar\lambda^2}{8}+{\bar\lambda^4} {\rm e}^{4\bar \lambda},\label{equ-est-ky}\\
\nonumber|k_{yy}(x,y)|^2\leq&\frac{3(32\bar \lambda'^{2}+31\bar \lambda^{4})}{256}\\
&+\frac{3\bar \lambda^2(32\bar\lambda^2+ 26\bar \lambda'^2+23\bar \lambda^4)e^{4\bar\lambda }}{128}, \label{equ-est-kyy}\\
\nonumber|k_{yyy}(x,y)|^2\leq& \frac{12\lambda''^2+73\bar\lambda'^2\bar\lambda^2+218\bar\lambda^4+20\bar \lambda^{6}}{64}\\
\nonumber&+\frac{\bar \lambda^2 {\rm e}^{4\bar \lambda }(36\bar \lambda'^{2}(1+\bar\lambda^2)+68\bar\lambda^4+19\bar \lambda''^2)}{16}\\
&+\frac{\bar \lambda^4e^{4\bar\lambda }(6\bar \lambda'^{2}+119\bar \lambda^{4})}{128}.\label{equ-est-kyyy}
\end{align}}
\end{lemm}

\begin{proof}\rm
Let us start by defining 
\begin{align}
G(\xi,\eta)=k(x,y),\label{equ-G-k}
\end{align}
where $\xi = x-y,\ \eta = x-y$. Then, 
\begin{align}
k_{y}(x,y)=G_\xi(\xi,\eta)-G_\eta(\xi,\eta),\label{equ-ky}
\end{align}
where
\begin{align}
G_\xi(\xi,\eta)=&-\frac{1}{4}\lambda\left(\frac{\xi}{2}\right)+\frac{1}{4}\int_0^\eta \lambda\left(\frac{\xi-s}{2}\right) G(\xi,s)\mathrm ds,\\
 \nonumber G_\eta(\xi,\eta)=&\frac{1}{4}\lambda\left(\frac{\eta}{2}\right)-\frac{1}{4}\int_0^\eta \lambda\left(\frac{\eta-s}{2}\right) G(\eta,s)\mathrm ds\\
&+\frac{1}{4}\int_\eta^\xi\lambda\left(\frac{\tau-\eta}{2}\right) G(\tau,\eta)\mathrm d\tau,
\end{align}
satisfy the following estimates 
\begin{align}
\nonumber&G_\xi(\xi,\eta)^2\\
\nonumber=&\bigg(-\frac{1}{4}\lambda\left(\frac{\xi}{2}\right)+\frac{1}{4}\int_0^\eta \lambda\left(\frac{\xi-s}{2}\right) G(\xi,s)\mathrm ds\bigg)^2\\
\nonumber\leq&\frac{\bar\lambda^2}{8}+\frac{\bar \lambda^2}{8}\int_0^\eta  G(\xi,s)^2\mathrm ds\\
\leq&\frac{\bar\lambda^2}{8}(1+ \bar \lambda ^2e^{4\bar \lambda}),\label{eq36}\\
 \nonumber& G_\eta(\xi,\eta)^{2}\\
 \nonumber=&\bigg(\frac{1}{4}\lambda\left(\frac{\eta}{2}\right)-\frac{1}{4}\int_0^\eta \lambda\left(\frac{\eta-s}{2}\right) G(\eta,s)\mathrm ds\\
\nonumber&+\frac{1}{4}\int_\eta^\xi\lambda\left(\frac{\tau-\eta}{2}\right) G(\tau,\eta)\mathrm d\tau\bigg)^2\\
\leq&\frac{3\bar\lambda^2}{16}+\frac{3\bar\lambda^2}{16}\int_0^\eta G(\eta,s)^{2}\mathrm ds+\frac{3\bar\lambda^2}{16}\int_\xi^\eta G(\tau,\eta)^2\mathrm d\tau\nonumber\\
\leq&\frac{3\bar\lambda^2}{16}(1+ 2\bar \lambda ^2e^{4\bar \lambda}).\label{eq37}
\end{align}
To arrive at \eqref{eq36}, \eqref{eq37}, we use the fact that $G(\xi,\eta)=k(x,y)$ and $|k(x,y)|\leq \bar \lambda {\rm e}^{2\bar \lambda x}$. Hence, based on \eqref{equ-ky}, \eqref{eq36}, and  \eqref{eq37}, one can get  
\begin{align}
\nonumber k_y(x,y)^2\leq& 2G_\xi(\xi,\eta)^2+2G_\eta(\xi,\eta)^2\\
\leq&\frac{5\bar\lambda^2}{8}+{\bar\lambda^4} {\rm e}^{4\bar \lambda},
\end{align}
is bounded. 

The next step is to find an upper bound of $k_{yy}(x,y)$. 
Knowing that 
\begin{align}
k_{yy}(x,y)=G_{\xi\xi}(\xi,\eta)-2G_{\xi\eta}(\xi,\eta)+G_{\eta\eta}(\xi,\eta),
\label{equ-kyy}
\end{align}
where
\begin{align}
G_{\eta\xi}(\xi,\eta)=&-\frac{1}{4}\lambda\left(\frac{\xi-\eta}{2}\right) G(\xi,\eta),\\
\nonumber G_{\xi\xi}(\xi,\eta)=&-\frac{1}{8}\lambda'\left(\frac{\xi}{2}\right)+\frac{1}{8}\int_0^\eta \lambda'\left(\frac{\xi-s}{2}\right) G(\xi,s)\mathrm ds\\
&+\frac{1}{4}\int_0^\eta \lambda\left(\frac{\xi-s}{2}\right) G_\xi(\xi,s)\mathrm ds,\\
 \nonumber G_{\eta\eta}(\xi,\eta)=&\frac{1}{8}\lambda'\left(\frac{\eta}{2}\right)-\frac{1}{8}\int_0^\eta \lambda'\left(\frac{\eta-s}{2}\right) G(\eta,s)\mathrm ds\\
\nonumber&-\frac{1}{4}\int_0^\eta \lambda\left(\frac{\eta-s}{2}\right) G_{\eta}(\eta,s)\mathrm ds\\
\nonumber&-\frac{1}{8}\int_\eta^\xi\lambda'\left(\frac{\tau-\eta}{2}\right) G(\tau,\eta)\mathrm d\tau\\
&+\frac{1}{4}\int_\eta^\xi\lambda \left(\frac{\tau-\eta}{2}\right) G_\eta(\tau,\eta)\mathrm d\tau,
\end{align}
the following estimates 
\begin{align}
G_{\eta\xi}(\xi,\eta)^{2}\leq&\frac{\bar \lambda^4 {\rm e}^{4\bar\lambda }}{16},\label{equ-G-etaxi}\\
G_{\xi\xi}(\xi,\eta)^{2}\leq&\frac{3(2\bar \lambda'^{2}+\bar \lambda^{4})}{128}+\frac{3\bar \lambda^2e^{4\bar\lambda }(2\bar \lambda'^{2}+\bar \lambda^{4})}{128},\label{equ-G-xixi}\\
 G_{\eta\eta}(\xi,\eta)^{2}\leq&\frac{5(4\bar \lambda'^{2}+5\bar\lambda^4)}{256}+\frac{5\bar \lambda^2e^{4\bar\lambda }(\bar \lambda'^2+\bar \lambda^4)}{32},\label{equ-G-etaeta}
\end{align}
can be deduced. Thus,  {from \eqref{equ-kyy} and \eqref{equ-G-etaxi}--\eqref{equ-G-etaeta}, } one gets that 
\begin{align}
|k_{yy}(x,y)|^2\leq 12G_{\eta\xi}(\xi,\eta)^{2}+3G_{\xi\xi}(\xi,\eta)^{2}+3G_{\eta\eta}(\xi,\eta)^2
\end{align}
is bounded.

The last step is to prove the boundedness of $k_{yy}(1,y)$ and $k_{yyy}(1,y)$.  
From \eqref{equ-kyy}, we derive the following relation
\begin{align}
\nonumber k_{yyy}(x,y)=&G_{\xi\xi\xi}(\xi,\eta)-3G_{\xi\xi\eta}(\xi,\eta)+3G_{\xi\eta\eta}(\xi,\eta)\\
&-G_{\eta\eta\eta}(\xi,\eta),
\end{align}
where
\begin{align}
\nonumber G_{\eta\eta\xi}(\xi,\eta)=&-\frac{1}{8}\lambda'\left(\frac{\xi-\eta}{2}\right) G(\xi,\eta)\\
&+\frac{1}{4}\lambda\left(\frac{\xi-\eta}{2}\right) G_\eta(\xi,\eta),\\
\nonumber G_{\eta\xi\xi}(\xi,\eta)=&\frac{1}{8}\lambda'\left(\frac{\xi-\eta}{2}\right) G(\xi,\eta)\\
&+\frac{1}{4}\lambda\left(\frac{\xi-\eta}{2}\right) G_\xi(\xi,\eta),\\
\nonumber G_{\xi\xi\xi}(\xi,\eta)=&-\frac{1}{16}\lambda''\left(\frac{\xi}{2}\right)+\frac{1}{4}\int_0^\eta \lambda\left(\frac{\xi-s}{2}\right) G_{\xi\xi}(\xi,s)\mathrm ds\\
\nonumber&+\frac{1}{4}\int_0^\eta \lambda'\left(\frac{\xi-s}{2}\right) G_\xi(\xi,s)\mathrm ds\\
&+\frac{1}{16}\int_0^\eta \lambda''\left(\frac{\xi-s}{2}\right) G(\xi,s)\mathrm ds,\\
 \nonumber  G_{\eta\eta\eta}(\xi,\eta)=&\frac{1}{16}\lambda''\left(\frac{\eta}{2}\right)-\frac{1}{16}\int_0^\eta \lambda''\left(\frac{\eta-s}{2}\right) G(\eta,s)\mathrm ds\\
\nonumber&-\frac{1}{4}\int_0^\eta \lambda'\left(\frac{\eta-s}{2}\right) G_\eta(\eta,s)\mathrm ds\\
\nonumber&-\frac{1}{4}\int_0^\eta \lambda\left(\frac{\eta-s}{2}\right) G_{\eta\eta}(\eta,s)\mathrm ds\\
\nonumber&+\frac{1}{8}\int_\eta^\xi \lambda''\left(\frac{\tau-\eta}{2}\right)G(\tau,\eta)\mathrm d\tau\\
\nonumber&-\frac{1}{4}\int_\eta^\xi\lambda'\left(\frac{\tau-\eta}{2}\right) G_\eta(\tau,\eta)\mathrm d\tau\\
\nonumber&+\frac{1}{4}\int_\eta^\xi\lambda\left(\frac{\tau-\eta}{2}\right) G_{\eta\eta}(\tau,\eta)\mathrm d\tau\\
&-\frac{1}{2}\lambda(0) G_\xi(\eta,\eta),
\end{align}
satisfy the following estimates 
\begin{align}
G_{\eta\eta\xi}(\xi,\eta)^2\leq&\frac{ 3\bar\lambda^4+2\bar \lambda^2 {\rm e}^{4\bar \lambda }(2\bar \lambda'^{2}+3\bar\lambda^4)}{128},\\
G_{\eta\xi\xi}(\xi,\eta)^{2}\leq& \frac{\bar\lambda^4+\bar \lambda^2 {\rm e}^{4\bar \lambda }(2\bar \lambda'^{2}+\bar\lambda^4)}{64},\\
\nonumber G_{\xi\xi\xi}(\xi,\eta)^{2}\leq&\frac{\bar \lambda''^2+2\bar\lambda'^2\bar\lambda^2}{64}+\frac{3\bar \lambda^2 (2\bar \lambda'^{2}+\bar \lambda^{4})}{512}\\
\nonumber&+\frac{\bar \lambda^2 {\rm e}^{4\bar \lambda }(\bar \lambda''^2+2\bar\lambda'^2\bar\lambda^2)}{64}\\
&+\frac{3\bar \lambda^4e^{4\bar\lambda }(2\bar \lambda'^{2}+\bar \lambda^{4})}{512},\\
 \nonumber G_{\eta\eta\eta}(\xi,\eta)^2\leq&\frac{\bar \lambda''^{2}+5\bar \lambda'^2\bar\lambda^2+16\bar\lambda^4}{32}+\frac{\bar\lambda^2 (44\bar \lambda'^{2}+37\bar \lambda^{4})}{512}\\
&+\frac{\bar \lambda^2e^{4\bar \lambda}(70\bar \lambda'^2\bar\lambda^2+36\bar\lambda''^{2}  +64\bar\lambda^4+29\bar \lambda^6)}{128}.
\end{align}
Hence, the boundedness of $k_{yyy}(x,y)$ is established through the inequality  
\begin{align}
\nonumber|k_{yyy}(x,y)|^2\leq& 36G_{\eta\eta\xi}(\xi,\eta)^{2}+36G_{\eta\xi\xi}(\xi,\eta)^{2}\\&+4G_{\xi\xi\xi}(\xi,\eta)^{2}+4G_{\eta\eta\eta}(\xi,\eta)^2,
\end{align} 
which completes the proof of  Lemma \ref{lemma-k}.
  \hfill$\blacksquare$
\end{proof}

\bigskip
\noindent{\bf Proof of Lemma \ref{lem2}.}
Next, we turn our attention to $\gamma(x,y)$. We first recall from \cite[Theorem 4.1]{karafyllis2018input} that, for every $k(1,\cdot)\in C^2([0,1]) \subset L^2(0,1)$, there exists a unique mapping $\gamma \in C^0([0,1];L^2(0,1))$, with $\gamma\in C^1((0,1]\times[0,1])$, satisfying $\gamma[x] \in C^2([0,1])$ for all $x\in(0,1]$, and \eqref{equ-gammax}--\eqref{equ-gamma0}.  

As the next step, we prove the boundedness of the gain kernel $\gamma(x,y)$. First, we use Agmon's inequality to obtain that the kernel $\gamma(x,y)$ satisfies the following estimate,
\begin{align}
|\gamma(x,y)|^2\leq& 2\int_0^1\gamma(x,y)\gamma_y(x,y)\mathrm dy\nonumber\\
\leq& \int_0^1\gamma(x,y)^2\mathrm dy+\int_0^1\gamma_y(x,y)^2\mathrm dy.\label{equ-gamma-bound}
\end{align}
With the inequalities
\begin{align}
\nonumber&\frac{\mathrm d}{\mathrm dx}\frac{1}{2}\int_0^1 \gamma(x,y  )^{2}\mathrm dy\\
\nonumber=&\int_0^1 \gamma(x,y)\gamma_x(x,y)\mathrm dy\\
\nonumber\leq&D \int_0^1 \gamma(x,y)(\gamma_{yy}(x,y)+\lambda(y)\gamma(x,y))\mathrm dy\\
\nonumber\leq&-D \int_0^1 \gamma_{y}(x,y)^2\mathrm dy+D\bar \lambda \int_0^1 \gamma(x,y)^2\mathrm dy\\
\leq&D\bar \lambda \int_0^1 \gamma(x,y)^2\mathrm dy,\label{equ-gamma-proof-0}
\end{align}
derived by using  integration by parts and   the comparison lemma, we arrive at the inequality
\begin{align}
\label{equ-gamma-part1-0}
\int_0^1 \gamma(x,y)^2\mathrm dy\leq {\rm e}^{2D\bar \lambda x}\int_0^1 \gamma(0,y  )^2\mathrm dy.
\end{align}
Knowing that $\gamma(0,y)=k(1,y)$, enables one to get the estimate below
\begin{align}
&\int_0^1 \gamma(x,y  )^2\mathrm dy\leq {\rm e}^{2D\bar \lambda x}\int_0^1k(1,y)^2\mathrm dy.\label{equ-gamma-part1}
\end{align}
Along the same arguments, the following holds
\begin{align}
\nonumber&\frac{\mathrm d}{\mathrm dx}\frac{1}{2}\int_0^1 \gamma_y(x,y)^{2}\mathrm dy\\
\nonumber=&\int_0^1 \gamma_{y}(x,y)\gamma_{xy}(x,y)\mathrm dy\\
\nonumber\leq&-D \int_0^1 \gamma_{yy}(x,y)(\gamma_{yy}(x,y)+\lambda(y)\gamma(x,y))\mathrm dy\\
\leq&D\bar \lambda \int_0^1 \gamma_{y}(x,y)^2\mathrm dy,
\end{align}
 and leads to the following inequality
\begin{align}
\int_0^1 \gamma_{y}(x,y  )^2\mathrm dy\leq {\rm e}^{2D\bar \lambda x}\int_0^1 \gamma_{y}(0,y  )^2\mathrm dy.
\end{align}
From \eqref{equ-gamma0}, it is obvious that $\gamma_{y}(0,y  )=k_{y}(1,y)$ and therefore
\begin{align}
\int_0^1 \gamma_{y}(x,y  )^2\mathrm dy\leq {\rm e}^{2D\bar \lambda x}\int_0^1k_{y}(1,y)^2\mathrm dy.\label{equ-gamma-part2}
\end{align}
Substituting  \eqref{equ-gamma-part1} and \eqref{equ-gamma-part2} into \eqref{equ-gamma-bound}, one can deduce that
\begin{align}\label{gamma-bound}
|\gamma(x,y)|^2\leq& 2\int_0^1\gamma(x,y)\gamma_y(x,y)\mathrm dy\nonumber\\
\leq& {\rm e}^{2D\bar \lambda x}\int_0^1k(1,y)^2\mathrm dy+e^{2D\bar \lambda x}\int_0^1k_{y}(1,y)^2\mathrm dy.
\end{align}
The first term of the right-hand side of \eqref{gamma-bound} is bounded due to  the boundedness of $k(x,y)$ (see. \eqref{equ-k-bouded}), which translates into $|k(1,y)|\leq \bar \lambda {\rm e}^{2\bar \lambda}$.  {Meanwhile, based on Lemma \ref{lemma-k}, we get that  $|k_y(x,y)|^2$ is bounded.} Thus, based on \eqref{gamma-bound}, we have proved the boundedness of $\gamma(x,y)$.

Finally, we  prove the boundedness of the gain kernel $q(x)$.
Since  $q(x) = -\gamma_y(x, 1)$ holds,  multiplying the PDE $\gamma_x(x,y)$ $=D \gamma_{yy}(x,y)+D\lambda(y)\gamma(x,y)$ by $2y\gamma_y(x,y)$, we get 
\begin{align}
\nonumber&2y\gamma_x(x,y)\gamma_y(x,y)\\
=&2D y\gamma_y(x,y)\gamma_{yy}(x,y)+2D\lambda(y)y\gamma(x,y)\gamma_y(x,y).\label{equ-eta-2y}
\end{align}
Integrating \eqref{equ-eta-2y} in $y$ we obtain the equality below
\begin{align}
\nonumber&2\int_0^1y\gamma_x(x,y)\gamma_y(x,y)\mathrm dy\\
\nonumber=&D \gamma_y(x,1)^2-D \int_0^1\gamma_y(x,y)^2\mathrm dy\\
&+2D\int_0^1 \lambda(y)y\gamma(x,y)\gamma_y(x,y) \mathrm dy.\label{equ-eta-2x}
\end{align}
Using Young's inequality, \eqref{equ-eta-2x} implies that
\begin{align}
\nonumber\gamma_y(x,1)^2\leq&\bar \lambda\int_0^1\gamma(x,y)^2\mathrm dy+\left(1+\bar \lambda+\frac{1}{D}\right)\int_0^1\gamma_y(x,y)^2\mathrm dy\\
&+\frac{1}{D}\int_0^1\gamma_x(x,y)^2\mathrm dy.\label{equ-etay00}
\end{align}
In order to find an estimate of the last term of the above equation, let us consider the following fact 
\begin{align}
\nonumber&\frac{\mathrm d}{\mathrm dx}\frac{1}{2}\int_0^1 \gamma_{x}(x,y  )^{2}\mathrm dy\\
\nonumber=&\int_0^1 \gamma_{x}(x,y)\gamma_{xx}(x,y)\mathrm dy\\
\nonumber\leq & D \int_0^1 \gamma_{x}(x,y)(\gamma_{xyy}(x,y)+\lambda(y)\gamma_x(x,y))\mathrm dy\\
\nonumber\leq &-D\int_0^1 \gamma_{xy}(x,y)^2\mathrm dy+D\int_0^1 \lambda(y)\gamma_x(x,y)^2\mathrm dy\\
\leq&D \bar \lambda \int_0^1 \gamma_{x}(x,y)^2\mathrm dy,\label{equ-gammax-proof0} 
\end{align}
where we have used integration by parts. Hence,  comparison lemma helps to deduce the following bound
\begin{align}\label{gam1}
\int_0^1 \gamma_{x}(x,y  )^2\mathrm dy\leq {\rm e}^{2D \bar \lambda x}\int_0^1 \gamma_{x}(0,y  )^2\mathrm dy.
\end{align}
From \eqref{equ-gammax} and \eqref{equ-gamma0}, it is obvious that $\gamma_x(0,y)=D\gamma_{yy}(0,y)+D\lambda(y)\gamma(0,y)$ and $\gamma_{yy}(0,y  )=k_{yy}(1,y),$  which enables one to rewrite \eqref{gam1} as 
\begin{align}
\int_0^1 \gamma_{x}(x,y  )^2\mathrm dy
\leq  2 D^2 {\rm e}^{2D\bar \lambda x}\int_0^1 (k_{yy}(1,y)^2+\bar \lambda^2 k(1,y)^2)\mathrm dy.\label{equ-gammax-proof1} 
\end{align}
% The next step is to find an upper bound of $k_{yy}(1,y)$. 
% Knowing that \begin{align}\label{equ-kyy}
% k_{yy}(x,y)=G_{\xi\xi}(\xi,\eta)-2G_{\xi\eta}(\xi,\eta)+G_{\eta\eta}(\xi,\eta),\end{align}
% where
% \begin{align}
% G_{\eta\xi}(\xi,\eta)=&-\frac{1}{4}\lambda\left(\frac{\xi-\eta}{2}\right) G(\xi,\eta),\\
% \nonumber G_{\xi\xi}(\xi,\eta)=&-\frac{1}{8}\lambda'\left(\frac{\xi}{2}\right)+\frac{1}{8}\int_0^\eta \lambda'\left(\frac{\xi-s}{2}\right) G(\xi,s)\mathrm ds\\
% &+\frac{1}{4}\int_0^\eta \lambda\left(\frac{\xi-s}{2}\right) G_\xi(\xi,s)\mathrm ds,\\
%  \nonumber G_{\eta\eta}(\xi,\eta)=&\frac{1}{8}\lambda'\left(\frac{\eta}{2}\right)-\frac{1}{8}\int_0^\eta \lambda'\left(\frac{\eta-s}{2}\right) G(\eta,s)\mathrm ds\\
% \nonumber&-\frac{1}{4}\int_0^\eta \lambda\left(\frac{\eta-s}{2}\right) G_{\eta}(\eta,s)\mathrm ds\\
% \nonumber&-\frac{1}{8}\int_\eta^\xi\lambda'\left(\frac{\tau-\eta}{2}\right) G(\tau,\eta)\mathrm d\tau\\
% &+\frac{1}{4}\int_\eta^\xi\lambda \left(\frac{\tau-\eta}{2}\right) G_\eta(\tau,\eta)\mathrm d\tau,
% \end{align}
% the following estimates 
% \begin{align}
% G_{\eta\xi}(\xi,\eta)^{2}\leq&\frac{\bar \lambda^4 {\rm e}^{4\bar\lambda }}{16},\\
% G_{\xi\xi}(\xi,\eta)^{2}\leq&\frac{3(2\bar \lambda'^{2}+\bar \lambda^{4})}{128}+\frac{3\bar \lambda^2e^{4\bar\lambda }(2\bar \lambda'^{2}+\bar \lambda^{4})}{128},\\
%  G_{\eta\eta}(\xi,\eta)^{2}\leq&\frac{5(4\bar \lambda'^{2}+5\bar\lambda^4)}{256}+\frac{5\bar \lambda^2e^{4\bar\lambda }(\bar \lambda'^2+\bar \lambda^4)}{32},
% \end{align}
%can be deduced. 
Thus,  {based on  the estimate of $|k_{yy}(x,y)|^2$ in Lemma \ref{lemma-k},} the  boundedness of $\gamma_{y}(x,1)$  follows from \eqref{equ-etay00},  which implies the boundedness of   $q(x)$.

\section{Proof of Proposition \ref{prop-0}}\label{lemma-3-proof}

To prove Proposition \ref{prop-0}, we first introduce the variable 
\begin{align}
\hat m(x,t)=\hat w(x,t)-xz(0,t),\label{equ-mw}
\end{align}
which leads to the following  scaled target  system with homogeneous boundary conditions
\begin{align}
\nonumber\hat m_t(x,t)=&\hat m_{xx}(x,t)-xz_t(0,t)+\delta_1(x)u(x,t)\\
&+\int_0^x\delta_2(x,y)u(y,t)\mathrm dy,\label{equ-m0}\\
\hat m(0,t)=&0,\\
\hat m(1,t)=&0,\\
D\hat z_t(x,t)=&\hat z_x(x,t)+\int_0^1\delta_3(x,y)u(y,t)\mathrm dy,\\
\hat z(1,t)=&0.\label{equ-z1-new}
\end{align}
With $A_1>0$ to be chosen, let us define the Lyapunov candidate
\begin{align}\label{V}
V_1(t)=\frac{A_1}{2}\int_0^1\hat m(x,t)^2\mathrm dx+\frac{D}{2}\int_0^1(1+x)(\hat z(x,t)+\hat z_x(x,t))\mathrm dx.
\end{align}
Computing the   time derivative of \eqref{V} along \eqref{equ-m0}--\eqref{equ-z1-new} as
\begin{align}
\nonumber\dot V_1(t)=&{A_1}\int_0^1\hat m(x,t)\hat m_t(x,t)\mathrm dx+{D}\int_0^1(1+x)\hat z(x,t)\hat z_t(x,t)\mathrm dx\\
\nonumber&+{D}\int_0^1(1+x)\hat z_x(x,t)\hat z_{xt}(x,t)\mathrm dx\\
\nonumber=&{A_{1}}\int_0^1\hat m(x,t)\bigg[\hat m_{xx}(x,t)-xz_t(0,t)+\delta_1(x)u(x,t)\\
\nonumber&+\int_0^x\delta_2(x,y)u(y,t)\mathrm dy\bigg]\mathrm dx+\int_0^1(1+x)\hat z(x,t)(\hat z_x(x,t)\\
\nonumber&+\int_0^1\delta_3(x,y)u(y,t)\mathrm dy)\mathrm dx+\int_0^1(1+x)\hat z_x(x,t)(\hat z_{xx}(x,t)\\
&+\int_0^1\delta_{3x}(x,y)u(y,t)\mathrm dy)\mathrm dx,
\end{align}
and using integration by parts and Young's inequality with $r_1>0$ yet to be chosen, the following estimate is  obtained:
\begin{align}
\nonumber\dot V_1(t)\leq&-A_{1}\rVert\hat m_{x}(t)\rVert^2+\frac{A_{1}}{D}\left(\frac{1}{2r_1}\rVert\hat m(t)\rVert^{2}+\frac{r_1}{6}\hat z_x(0,t)^2\right)\\
\nonumber&-\frac{1}{2}\hat z(0,t)^2-\frac{1}{2}\rVert\hat z(t)\rVert^2-\frac{1}{2}\hat z_x(0,t)^2-\frac{1}{2}\rVert\hat z_x(t)\rVert^2\\
\nonumber&+A_1\int_0^1\hat m(x,t)\delta_1(x)u(x,t)\mathrm dx\\
\nonumber&+A_1\int_0^1\hat m(x,t)\int_0^x\delta_2(x,y)u(y,t)\mathrm dy\mathrm dx\\
\nonumber&+\left(\int_0^1\delta_3(1,y)u(y,t)\mathrm dy\right)^2\\
\nonumber&-\frac{A_1}{D}\int_0^1\hat m(x,t)x\int_0^1\delta_3(0,y)u(y,t)\mathrm dy\mathrm dx\\
\nonumber&+\int_0^1(1+x)\hat z(x,t)\int_0^1\delta_3(x,y)u(y,t)\mathrm dy\mathrm dx\\
\nonumber&+\int_0^1(1+x)\hat z_x(x,t)\int_0^1\delta_{3x}(x,y)u(y,t)\mathrm dy\mathrm dx\\
\nonumber\leq&-A_1\rVert\hat m_{x}(t)\rVert^2+\frac{A_1}{2Dr_1}\rVert\hat m(t)\rVert^{2}+\frac{A_1r_1}{6D}\hat z_x(0,t)^2\\
\nonumber&-\frac{1}{2}\hat z(0,t)^2-\frac{1}{2}\rVert\hat z(t)\rVert^2-\frac{1}{2}\hat z_x(0,t)^2-\frac{1}{2}\rVert\hat z_x(t)\rVert^2\\
\nonumber&+A_1\epsilon( \rVert\hat m(t)\rVert^{2}+\rVert u(t)\rVert^2)+\epsilon^{2} \rVert u(t)\rVert^{2}\\
\nonumber&+\frac{A_1\epsilon}{2D}(\rVert\hat m(t)\rVert^2+\rVert u(t)\rVert^{2})\\
&+\epsilon(\rVert\hat z(t)\rVert^2+\rVert\hat z_x(t)\rVert^{2}+2\rVert u(t)\rVert^2).
\end{align}
%\textcolor{violet}{where $r_1>0$.} 
Hence, by virtue of  Poincar\'e inequality, we get that
\begin{align}
\nonumber\dot V_1(t)&\leq-A_1\left(\frac{1}{4}-\frac{1}{2Dr_1}-\epsilon-\frac{\epsilon}{2D}\right)\rVert\hat m(t)\rVert^2-\frac{1}{2}\hat z(0,t)^2\\
\nonumber&-\left(\frac{1}{2}-\epsilon\right)(\rVert\hat z(t)\rVert^2+\rVert\hat z_x(t)\rVert^2)-\left(\frac{1}{2}-\frac{A_1r_1}{6D}\right)\hat z_x(0,t)^2\\
&+\epsilon\left(A_1+\epsilon +2+\frac{A_1}{2D}\right)\rVert u(t)\rVert^{2}.\label{equ-V-dot}
\end{align}
Following \cite{MKrstic2009}, the inverse transformation of the approximate gain kernel  \eqref{approx-trans-w} allows to derive a bound of the norm of the state $u(x,t)$ in \eqref{equ-V-dot} with respect to the norm of the approximate target system's state $\hat w(x,t)$. In other words,
\begin{align}\label{approx-inverse0}
&u(x,t)=\hat w(x,t)+\int_0^x\hat l(x,y)\hat w(y,t)\mathrm dy,
\end{align}
where the inverse kernel $\hat l(x,y)$ and  the kernel $\hat k(x,y)$ 
satisfy  the following equation \cite{Krstic2008book}
\begin{align}
\hat l(x,y)=\hat k(x,y)+\int_y^x\hat k(x,\xi)\hat l(\xi,y)\mathrm d\xi,
\end{align}
and the following conservative bound holds,
\begin{align}
\rVert \hat l\rVert_\infty\leq \rVert \hat k\rVert_\infty {\rm e}^{\rVert \hat k\rVert_\infty}.
\end{align}
Since 
\begin{align}
\rVert k-\hat k\rVert_\infty<\epsilon,
\end{align}
it follows  that 
\begin{align}
\rVert\hat k\rVert_\infty\leq\rVert k\rVert_\infty+\epsilon
\end{align}
and using \eqref{equ-k-bouded}, we obtain the estimate  
\begin{align}
\rVert\hat k\rVert_\infty\leq \bar \lambda {\rm e}^{2\bar \lambda }+\epsilon.
\end{align}
Hence, 
\begin{align}\label{eq88}
&\rVert \hat l\rVert_\infty\leq(\lambda {\rm e}^{2\bar \lambda x}+\epsilon)e^{(\lambda {\rm e}^{2\bar \lambda x}+\epsilon)}.
\end{align}
From \eqref{equ-mw} and \eqref{approx-inverse0}, we deduce  the following relation 
\begin{align}
\nonumber\rVert u(t)\rVert^2 =&\int_0^1\bigg(\hat w(x,t)+\int_0^x\hat l(x,y)\hat w(y,t)\mathrm dy\bigg)^2\mathrm dx\\
\nonumber \leq&2(1+\rVert \hat l\rVert_\infty^{2})\rVert \hat w(t)\rVert ^2\\
\leq&4(1+\rVert \hat l\rVert_\infty^{2})(\rVert \hat m(t)\rVert ^2+z(0,t)^2).\label{equ-uw}
\end{align}
Substituting  \eqref{equ-uw} into \eqref{equ-V-dot} yields  the following bound
\begin{align}\label{equ-V-00}
\nonumber\dot V_1(t)\leq&-\bigg(\frac{A_1}{8}-\frac{A_1}{2Dr_1}-\epsilon^2(1+\rVert \hat l\rVert_\infty^{2})\bigg)\rVert \hat m(t)\rVert ^2\\
\nonumber&-\bigg(\frac{A_1}{8}-\epsilon\bigg(4(A_1+2+\frac{A_1}{2D})(1+\rVert \hat l\rVert_\infty^{2})\\
\nonumber&+A_1+\frac{A_1}{2D}\bigg)\bigg)\rVert \hat m(t)\rVert ^2-\left(\frac{1}{2}-\epsilon\right)\rVert\hat z(t)\rVert^2\\
\nonumber&-\left(\frac{1}{4}-4\epsilon^{2}(1+\rVert \hat l\rVert_\infty^{2})\right)z(0,t)^2\\
\nonumber&-\left(\frac{1}{4}-4\epsilon\left(A_1+2+\frac{A_1}{2D}\right)(1+\rVert \hat l\rVert_\infty^{2})\right)z(0,t)^2\\
&-\left(\frac{1}{2}-\epsilon\right)\rVert\hat z_x(t)\rVert^2-\left(\frac{1}{2}-\frac{A_1r_1}{6D}\right)\hat z_x(0,t)^2,
\end{align}
for the time-derivative of the Lyapunov function $V(t)$. Setting  \begin{align}
\nonumber A_1=\frac{3D^2}{8}, \quad r_1 =\frac{8}{D},
\end{align}
one can define $\varepsilon^*(D)$ as
\begin{align}
\nonumber\epsilon^*=&\min\bigg\{\frac{D}{8(2A_1D+4D+{A_1})(1+\rVert \hat l\rVert_\infty^{2})},\ \frac{1}{2},\\
\nonumber&\frac{A_1D}{16(2A_1D+4D+{A_1})(1+\rVert \hat l\rVert_\infty^{2})+4A_1(2D+1))},\\
&\frac{1}{\sqrt{(1+\rVert \hat l\rVert_\infty^{2})}},\ \sqrt{\frac{A_1(Dr_1-4)}{8Dr_1(1+\rVert \hat l\rVert_\infty^{2})}}\bigg\},
\end{align}
such that for all  $\epsilon\in (0,\epsilon^*)$,
\begin{align}
\dot V_{1}(t)\leq-c_{1}V_{1}(t),
\end{align}
where $c_1(D,\varepsilon)$ is defined by
\begin{align}
\nonumber c_{1}=&\min\bigg\{\bigg[\frac{1}{2}-\frac{1}{Dr_1}-2\epsilon\bigg(4\bigg(1+\frac{\epsilon+2}{A_1} +\frac{1}{2D}\bigg)(1+\rVert \hat l\rVert_\infty^{2})\\
&+1+\frac{1}{2D}\bigg)\bigg],\ \frac{1-2\epsilon}{2D}\bigg\},
\end{align}
which leads to the following inequality \begin{align}
V_{1}(t)\leq V_{1}(0)e^{-c_1 t}.\end{align} 
% Defining   
% \begin{align}
% \Psi(t)=\rVert\hat w(t)\rVert^2+\rVert\hat z(t)\rVert^2+\rVert\hat z_x(t)\rVert^2,\label{equ-Psi}
%\end{align}
One can readily verify that 
\begin{align}
\frac{1}{A_{1}+D}V_{1}(t)\leq\Psi_1(t)\leq\max\left\{\frac{4}{A_{1}},\ \frac{6}{D}\right\}V_{1}(t).
\end{align}
Therefore, for all {$\epsilon\in (0,\epsilon^*)$}, the  exponential stability bound \eqref{eq-Psi-exp-bound} holds
% \begin{align}
% \Psi(t)\leq \Psi(t_0)c_2 {\rm e}^{-c_1 (t-t_0)},
% \end{align}
with 
\begin{align}
c_2(D)=(A_{1}+D)\max\left\{\frac{4}{A_{1}},\ \frac{6}{D}\right\}.
\end{align}

\section{Proof of Proposition \ref{propo1}}\label{appendix-B}

Proposition \ref{propo1} is proven using the following lemmas.

\begin{lemm} \label{lemma3}
For the kernel $\gamma$ that satisfies  the PDE \eqref{equ-gammax}--\eqref{equ-gamma0}, 
% \begin{align}
% \gamma_x(x,y  )=&D \gamma_{yy}(x,y  )+D \lambda(y)\gamma(x,y  ),\label{equ-lambda-new0}\\
% \gamma(x,1)=&0,\\
% \gamma(x,0  )=&0,\\
% \gamma(0,y) = &k(1,y),\label{equ-lambda-new1}
% \end{align}
 the following estimates hold: 
\begin{align}
\label{equ-lambda-inq1}
&\int_0^1\int_0^1\gamma(x,y  )^2\mathrm dy\mathrm dx\leq\frac{{\rm e}^{2D\bar \lambda}}{2D\bar \lambda }\int_0^1k(1,y)^2\mathrm dy,\\
\label{equ-lambda-inq2}
\nonumber&\int_0^1\int_0^1\gamma_x(x,y  )^2 \mathrm dy\mathrm dx\\
\leq& \frac{De^{2D\bar \lambda }}{\bar \lambda}\int_0^1 (k_{yy}(1,y)^2+\bar \lambda ^{2}k(1,y)^2)\mathrm dy.
\end{align}
\end{lemm}

\begin{proof} \rm
% \textcolor{red}{From \eqref{equ-gamma-proof-0}--\eqref{equ-gamma-part2}, the fact is follows
% \begin{align}
% &\int_0^1 \gamma(x,y  )^2\mathrm dy\leq {\rm e}^{2D\bar \lambda x}\int_0^1k(1,y)^2\mathrm dy.
% \label{equ-gamma-part1-0}
% \end{align}
By integrating both sides of \eqref{equ-gamma-part1-0} 
with respect to $x$ from $0$ to $1$, 
we  get
\begin{align}
\nonumber&\int_0^1\int_0^1 \gamma(x,y)^2\mathrm dy\mathrm dx\\
\leq&\int_0^1\bigg(e^{2D\bar \lambda x}\int_0^1k(1,y)^2\mathrm dy\bigg)\mathrm dx\nonumber\\
\leq& \frac{{\rm e}^{2D\bar \lambda}}{2D\bar \lambda }\int_0^1k(1,y)^2\mathrm dy.
\end{align}
Meanwhile, following the proof of \eqref{equ-gammax-proof0}--\eqref{equ-gammax-proof1}, and take integrate \eqref{equ-gammax-proof1} with respect to $x$ from $0$ to $1$, we can obtain the estimate below
\begin{align}
\nonumber&\int_0^1\int_0^1 \gamma_{x}(x,y)^2\mathrm dy\mathrm dx\\
\nonumber\leq& \int_0^1\bigg(2 D^2 {\rm e}^{2D\bar \lambda x}\int_0^1 (k_{yy}(1,y)^2+\bar \lambda^2 k(1,y)^2)\mathrm dy\bigg)\mathrm dx\\
\leq &\frac{De^{2D\bar \lambda }}{\bar \lambda}\int_0^1 (k_{yy}(1,y)^2+\bar \lambda ^{2}k(1,y)^2)\mathrm dy,\label{equ-gamma-x-part2}
\end{align}
which completes the proof.
%of  Lemma \ref{lemma3}.
  \hfill$\blacksquare$
\end{proof}

\begin{lemm} \label{lemma4}
If the kernel $\gamma$ satisfies  the reaction-diffusion PDE \eqref{equ-gammax}--\eqref{equ-gamma0}
%\eqref{equ-lambda-new0}-\eqref{equ-lambda-new1},
the following inequalities hold: 
\begin{align}
\label{equ-gamma-inq3}
\nonumber&\int_0^1 \gamma_y(x,1)^2\mathrm dx\\
\nonumber\leq&\int_0^1\bigg(\frac{{\rm e}^{2D\bar \lambda }}{\bar \lambda}k_{yy}(1,y)^2+\left(1+\bar \lambda +\frac{1}{D}\right)\frac{{\rm e}^{2D\bar \lambda }}{2D\bar \lambda}k_{y}(1,y)^{2}\\
&+\left(\frac{1}{2D}+\bar \lambda\right) {{\rm e}^{2D\bar \lambda }}k(1,y)^2\bigg)\mathrm dy,\\
\label{equ-gamma-inq4}
\nonumber &\int_0^1\gamma_{xy}(x,1)^2\mathrm dx\\
\nonumber=&3D^{2}\int_0^1 (k_{yyy}(1,y)^{2}+\bar \lambda^2k_{y}(1,y)^2+\bar \lambda'^2k(1,y)^2)\mathrm dy\\
\nonumber&+\bigg(\left(2D+D\bar \lambda +1+2D^{2}\left(\frac{\bar \lambda'}{\bar \lambda}+2\bar \lambda+\bar \lambda'\right)\right){{\rm e}^{2D\bar \lambda }}+2D^{2}\bar \lambda\\
&+2D^{2}\bar \lambda'+D+D\bar \lambda +1\bigg)\int_0^1 \left(k_{yy}(1,y)^2+\bar \lambda^2k(1,y)\right)\mathrm dy.
\end{align}
\end{lemm}

\begin{proof}\rm
 To prove \eqref{equ-gamma-inq3} and \eqref{equ-gamma-inq4}, the first step is building the relationship between $\gamma_{y}(x,1),\ \gamma_{xy}(x,1)$ and $k(1,y)$, and the second step is to prove the boundedness of $k_{yy}(1,y)$ and $k_{yyy}(1,y)$. We recall that we have proved the boundness of $k_y(1,y)$ and $k(1,y )$ in Section \ref{sec3}. 

To complete the  first step,  following the method in \eqref{equ-eta-2y}--\eqref{equ-etay00}, multiply the PDE $\gamma_{xx}(x,y)=D\gamma_{xyy}(x,y  )+D\lambda(y)\gamma_x(x,y)$ by $2y\gamma_{xy}(x,y  )$, which results into the following relation
\begin{align}
\nonumber&2y\gamma_{xx}(x,y)\gamma_{xy}(x,y)\\
=&2D y\gamma_{xy}(x,y)\gamma_{xyy}(x,y)+2D\lambda(y)y\gamma_x(x,y)\gamma_{xy}(x,y).\label{equ-gamma-2y}
\end{align}
Integrating \eqref{equ-gamma-2y} in $y$ and using integration by parts, one gets
\begin{align}
\nonumber&2\int_0^1y\gamma_{xx}(x,y)\gamma_{xy}(x,y)\mathrm dy\\
\nonumber=&D \gamma_{xy}(x,1)^2-D \int_0^1\gamma_{xy}(x,y)^2\mathrm dy\\
&+2D\int_0^1\lambda(y)y\gamma_x(x,y)\gamma_{xy}(x,y)\mathrm dy.
\end{align}
With the help of  Young's inequality, the following holds
\begin{align}
\nonumber \gamma_{xy}(x,1)^2=&\frac{1}{D}\int_0^1\gamma_{xx}(x,y)^2\mathrm dy+\bigg(1+\bar \lambda +\frac{1}{D}\bigg)\int_0^1\gamma_{xy}(x,y)^2\mathrm dy\\
&+\bar \lambda \int_0^1\gamma_x(x,y)^2\mathrm dy.\label{equ-gamma-xy0}
\end{align}
Integrating \eqref{equ-etay00} and \eqref{equ-gamma-xy0} in $x$, one arrives at
\begin{align}
\nonumber&\int_0^1 \gamma_y(x,1)^2\mathrm dx\\
\nonumber=&\bar \lambda \int_0^1\int_0^1\gamma(x,y)^2\mathrm dy\mathrm dx+\frac{1}{D}\int_0^1\int_0^1\gamma_x(x,y)^2\mathrm dy\mathrm dx\\
&+\left(1+\bar \lambda +\frac{1}{D}\right) \int_0^1\int_0^1\gamma_y(x,y)^2\mathrm dy\mathrm dx,\\
\nonumber &\int_0^1\gamma_{xy}(x,1)^2\mathrm dx\\
\nonumber=&\frac{1}{D}\int_0^1\int_0^1\gamma_{xx}(x,y)^2\mathrm dy\mathrm dx+\bar \lambda \int_0^1\int_0^1\gamma_x(x,y)^2\mathrm dy\mathrm dx\\
&+\left(1+\bar \lambda +\frac{1}{D}\right)\int_0^1\int_0^1\gamma_{xy}(x,y)^2\mathrm dy\mathrm dx.
\end{align}
Knowing that
\begin{align}
\nonumber&\frac{\mathrm d}{\mathrm dx}\frac{1}{2}\int_0^1 \gamma_{x}(x,y  )^{2}\mathrm dy\\
\nonumber=&\int_0^1 \gamma_{x}(x,y)\gamma_{xx}(x,y)\mathrm dy\\
\nonumber=&D \int_0^1 \gamma_{x}(x,y)(\gamma_{xyy}(x,y)+\lambda(y)\gamma_x(x,y))\mathrm dy\\
=&-D  \int_0^1 \gamma_{xy}(x,y)^2\mathrm dy+D \int_0^1 \lambda(y)\gamma_{x}(x,y)^{2}\mathrm dy,\label{equ-gammaxy}
\end{align}
where integration by parts is used, and then, integrating \eqref{equ-gammaxy} in $x$,   the following inequality is deduced
\begin{align}
\nonumber&\int_0^1\int_0^1 \gamma_{xy}(x,y)^2\mathrm dy\mathrm dx\\
\leq& \frac{1}{2D}\int_0^1 \gamma_{x}(0,y  )^2\mathrm dy+\bar \lambda\int_0^1\int_0^1 \gamma_{x}(x,y)^2\mathrm dy\mathrm dx.
\end{align}
From \eqref{equ-gammax} and \eqref{equ-gamma0}, it is obvious that $\gamma_x(0,y)=D\gamma_{yy}(0,y)+D\lambda(y)\gamma(0,y)$ and $\gamma_{yy}(0,y)=k_{yy}(1,y)$, therefore
\begin{align}
\nonumber&\int_0^1\int_0^1 \gamma_{xy}(x,y)^2\mathrm dy\mathrm dx\\
\nonumber\leq& D\int_0^1 (k_{yy}(1,y)^2+\bar \lambda^2k(1,y))\mathrm dy+\bar \lambda\int_0^1\int_0^1 \gamma_{x}(x,y)^2\mathrm dy\mathrm dx\\
\leq& D(1+e^{2D\bar \lambda})\int_0^1 (k_{yy}(1,y)^2+\bar \lambda^2k(1,y))\mathrm dy.\label{equ-gamma-xy}
\end{align}
Let us integrate by parts  the equality below
\begin{align}
\nonumber&\frac{\mathrm d}{\mathrm dx}\frac{1}{2}\int_0^1 \gamma_{xy}(x,y  )^{2}\mathrm dy\\
\nonumber=&\int_0^1 \gamma_{xy}(x,y)\gamma_{xxy}(x,y)\mathrm dy\\
\nonumber=&-D \int_0^1 \gamma_{xyy}(x,y)(\gamma_{xyy}(x,y)+\lambda(y)\gamma_x(x,y))\mathrm dy\\
\nonumber=&-D  \int_0^1 \gamma_{xyy}(x,y)^2\mathrm dy+D \int_0^1 \lambda(y)\gamma_{xy}(x,y)^{2}\mathrm dy\\
&+D \int_0^1 \lambda'(y)\gamma_{x}(x,y)\gamma_{xy}(x,y)\mathrm dy.\label{equ-gammaxy-ineq}
\end{align}
Hence,  taking the integral of \eqref{equ-gammaxy-ineq} with respect to $x$, the following holds  
\begin{align}
\nonumber&\int_0^1\int_0^1 \gamma_{xyy}(x,y)^2\mathrm dy\mathrm dx\\
\nonumber\leq& \frac{1}{2D}\int_0^1 \gamma_{xy}(0,y  )^2\mathrm dy+(\bar \lambda+\bar \lambda')\int_0^1\int_0^1 \gamma_{xy}(x,y)^2\mathrm dy\mathrm dx\\
\nonumber&+\bar \lambda'\int_0^1\int_0^1 \gamma_{x}(x,y)^2\mathrm dy\mathrm dx,\\
\nonumber\leq &\frac{3D}{2}\int_0^1 (k_{yyy}(1,y)^{2}+\bar \lambda^2k_{y}(1,y)^2+\bar \lambda'^2k(1,y)^2)\mathrm dy\\
\nonumber&+\bigg(\bar \lambda'\frac{De^{2D\bar \lambda }}{\bar \lambda}+D(\bar \lambda+\bar \lambda')(1+e^{2D\bar \lambda})\bigg)\int_0^1 (k_{yy}(1,y)^2\\
&+\bar \lambda^2k(1,y))\mathrm dy.
\end{align}
Based on \eqref{equ-gammax} and \eqref{equ-gamma0}, we get that 
\begin{align}
\gamma_{xy}(0,y)=&D\gamma_{yyy}(0,y)+D\lambda(y)\gamma_{y}(0,y)+D\lambda'(y)\gamma(0,y),
\end{align}
and 
\begin{align}
\gamma_{yyy}(0,y)=&k_{yyy}(1,y).
\end{align}
Since 
\begin{align}
\gamma_{xx}(x,y)=D\gamma_{xyy}(x,y)+D\lambda(y)\gamma_x(x,y),
\end{align}
it follows that
\begin{align}
\nonumber&\int_0^1\int_0^1\gamma_{xx}(x,y)^2\mathrm dy\mathrm dx\\
\nonumber\leq& 2D^2\int_0^1\int_0^1(\gamma_{xyy}(x,y)^{2}+\bar \lambda^2\gamma_x(x,y)^{2})\mathrm dy\mathrm dx\\
\nonumber\leq&3D^{3}\int_0^1 (k_{yyy}(1,y)^{2}+\bar \lambda^2k_{y}(1,y)^2+\bar \lambda'^2k(1,y)^2)\mathrm dy\\
\nonumber&+2D^{3}\bigg(\frac{\bar \lambda'}{\bar \lambda}+2\bar \lambda+\bar \lambda'){{\rm e}^{2D\bar \lambda }}+\bar \lambda+\bar \lambda'\bigg)\int_0^1 (k_{yy}(1,y)^2\\
&+\bar \lambda^2k(1,y)^2)\mathrm dy,\label{equ-gamma-xx}
\end{align}
and recalling \eqref{equ-gamma-x-part2} leads to the estimate
\begin{align}
\int_0^1\int_0^1 \gamma_{y}(x,y  )^2\mathrm dy\mathrm dx\leq \frac{{\rm e}^{2D\bar \lambda }}{2D\bar \lambda}\int_0^1k_{y}(1,y)^2\mathrm dy.\label{equ-gammay2}
\end{align}
Combining \eqref{equ-gamma-xy}, \eqref{equ-gamma-xx},  \eqref{equ-gammay2}, with Lemma \ref{lemma3}, inequalities  \eqref{equ-gamma-inq3} and \eqref{equ-gamma-inq4} are  derived.
 { Based on Lemma \ref{lemma-k},}  
 the boundedness of $k_{yyy}(x,y)$ is established, which completes the proof of  Lemma \ref{lemma4}.
  \hfill$\blacksquare$
\end{proof}

\bigskip
\noindent{\bf Proof of Proposition \ref{propo1}.} We derive the following estimate of   the $L^2$ norm of $\hat w$:
\begin{align}
\nonumber&\int_0^1\hat w(x,t)^2\mathrm dx\\
\nonumber=&\int_0^1\bigg(u(x,t)-\int_0^x\hat k(x,y)u(y,t)\mathrm dy\bigg)^2\mathrm dx\\
\nonumber\leq &2\int_0^1 u(x,t)^2\mathrm dx+2\int_0^1\left(\int_0^x\hat k(x,y)u(y,t)\mathrm dy\right)^2\mathrm dx\\
\nonumber\leq &2\rVert u(t)\rVert^2+2\int_0^1\int_0^x\hat k(x,y)^2\mathrm dy\mathrm dx\rVert u(t)\rVert^2\\
\leq &2\left(1+\int_0^1\int_0^1\hat k(x,y)^2\mathrm dy\mathrm dx\right)\rVert u(t)\rVert^2,\label{norm_w}
\end{align}
where we use Cauchy-Schwarz inequality.  Then, the $L^2$ norm of $v$ satisfies
\begin{align}
\nonumber&\int_0^1\hat z(x,t)^2\mathrm dx\\
\nonumber=&\int_0^1\left(v(x,t)-\int_0^1\hat \gamma(x,y)u(y,t)\mathrm dy\right.\\
\nonumber&-\left.D \int_0^x\hat q(x-y  )v(y,t)\mathrm dy\right)^2\mathrm dx\\
\nonumber\leq&3\int_0^1v(x,t)^2\mathrm dx+3\int_0^1\left(\int_0^1\hat \gamma(x,y  )u(y,t)\mathrm dy\right)^2\mathrm dx\\
\nonumber&+3D^2\int_0^1\left(\int_0^x\hat q(x-y)v(y,t)\mathrm dy\right)^2\mathrm dx\\
\nonumber\leq& 3\int_0^1\int_0^1\hat \gamma(x,y)^{2}\mathrm dy\mathrm dx\rVert u(t)\rVert^2\\
&+3\left(1+D^2\int_0^1\int_0^x\hat q(x-y)^2\mathrm dy\mathrm dx\right)\rVert v(t)\rVert^2,
\end{align}
 where we use Cauchy-Schwarz inequality.
% and $p(x,y  )=-\eta_y(x-y,1  )$. 
% Since
% \begin{align}
% \nonumber&\int_0^1\int_0^x\hat \eta_y(x-y,1  )^2\mathrm dy\mathrm dx\\
% \nonumber=&\int_0^1\int_0^x\hat \eta_y(s,1  )^{2}\mathrm ds\mathrm dx\leq\int_0^1\int_0^1\hat \eta_y(s,1  )^{2}\mathrm ds\mathrm dx\\
% =&\int_0^1\hat \eta_y(x,1  )^{2}\mathrm dx, 
% \end{align}
% we get 
% \begin{align}
% \nonumber\int_0^1v(x,t)^{2}\mathrm dx\leq &3\int_0^1\int_0^1\hat \eta(x,y  )^{2}\mathrm dy\mathrm dx\rVert\hat   w\rVert^2\\
% &+3\left(1+D^2 \int_0^1\hat \eta_y(x,1  )^{2}\mathrm dx\right)\rVert\hat  z\rVert^2.\label{norm_v}
%\end{align}
%% Since
% \begin{align}
% \nonumber&\int_0^1\int_0^x\eta_{xy}(x-y,1  )^2\mathrm dy\mathrm dx\\
% \nonumber=&\int_0^1\int_0^x\eta_{xy}(s,1  )^2\mathrm ds\mathrm dx\\
% \nonumber\leq&\int_0^1\int_0^1\eta_{xy}(s,1  )^2\mathrm ds\mathrm dx\\
% =&\int_0^1\eta_{xy}(x,1D  )^2\mathrm dx, 
% \end{align}
% we have

Taking  the first derivative of $v(x,t)$ with respect to $x$, we get the following estimates
\begin{align}
&\nonumber\int_0^1\hat z_x(x,t)^{2}\mathrm dx\\
\nonumber=&\int_0^1\bigg(v_{x}(x,t)-\int_0^1\hat \gamma_{x}(x,y)u(y,t)\mathrm dy-D \hat q(0)v(x,t)\\ 
\nonumber&-D \int_0^x\hat q'(x-y  )v(y,t)\mathrm dy\bigg)^2\mathrm dx\\
\nonumber\leq&4\int_0^1v_{x}(x,t)^{2}\mathrm dx+4\int_0^1\left(\int_0^1\hat \gamma_{x}(x,y  )u(y,t)\mathrm dy\right)^2\mathrm dx\\
\nonumber&+4D^2\int_0^1(\hat q(0)v(x,t))^2\mathrm dx\\
\nonumber&+4D^2\int_0^1\left(\int_0^x\hat q'(x-y)v(y,t)\mathrm dy\right)^2\mathrm dx\\
\nonumber\leq&4\int_0^1\int_0^1\hat \gamma_x(x,y)^{2}\mathrm dy\mathrm dx \rVert u(t)\rVert^2+4D^2\bigg(\hat q(0)^2\\
&+\int_0^1\int_0^x\hat q'(x-y )^{2}\mathrm dy\mathrm dx\bigg)\rVert v(t)\rVert^2+4\rVert v_{x}(t)\rVert^2,\label{138}
\end{align}
where we use Cauchy-Schwarz inequality.
Hence, using  \eqref{norm_w}--\eqref{138}, we derive  
\begin{align}
\nonumber S_1=&9+2\int_0^1\int_0^x\hat k(x,y)^2\mathrm dx\mathrm dy+3\int_0^1\int_0^1\hat \gamma(x,y)^2\mathrm dx\mathrm dy\\
\nonumber&+4\int_0^1\int_0^1\hat \gamma_x(x,y)^2\mathrm dx\mathrm dy+3D^{2}\int_0^1\int_0^x\hat q(x-y)^2\mathrm dx\mathrm dy\\
&+4D^{2}\hat q(0)^2+4D^{2}\int_0^1\int_0^x\hat q'(x-y)^2\mathrm dx\mathrm dy.\label{139}
\end{align}
Since $\hat \chi=\chi-\tilde \chi$, $S_1$ in \eqref{139} satisfies the following inequalities given  with respect to the error variables
\begin{align}
\nonumber S_1=&9+2\int_0^1\int_0^x(k(x,y)-\tilde k(x,y))^2\mathrm dx\mathrm dy+3\int_0^1\int_0^1(\gamma(x,y)\\
\nonumber&-\tilde \gamma(x,y))^2\mathrm dx\mathrm dy+4\int_0^1\int_0^1(\gamma_x(x,y)-\tilde \gamma_x(x,y))^2\mathrm dx\mathrm dy\\
\nonumber&+3D^{2}\int_0^1\int_0^x(q(x-y)-\tilde q(x-y))^2\mathrm dx\mathrm dy+4D^{2}(q(0)\\
\nonumber&-\tilde q(0))^2+4D^{2}\int_0^1\int_0^x(q'(x-y)-\tilde q'(x-y))^2\mathrm dx\mathrm dy\\
\nonumber \leq&9+4\int_0^1\int_0^x(k(x,y)^{2}+\epsilon^2)\mathrm dx\mathrm dy+6\int_0^1\int_0^1(\gamma(x,y)^2\\
\nonumber&+\epsilon^2)\mathrm dx\mathrm dy+8\int_0^1\int_0^1(\gamma_x(x,y)^{2}+\epsilon^2)\mathrm dx\mathrm dy\\
\nonumber&+6D^{2}\int_0^1\int_0^x(q(x-y)^{2}+\epsilon^2)\mathrm dx\mathrm dy+8D^{2}(q(0)^{2}\\
\nonumber&+\epsilon^2)+8D^{2}\int_0^1\int_0^x(q'(x-y)^{2}+\epsilon^2)\mathrm dx\mathrm dy\\
\nonumber \leq&9+(18+22D^{2})\epsilon^2+4\int_0^1\int_0^xk(x,y)^{2}\mathrm dx\mathrm dy\\
\nonumber&+6\int_0^1\int_0^1\gamma(x,y)^2\mathrm dx\mathrm dy+8\int_0^1\int_0^1\gamma_x(x,y)^{2}\mathrm dx\mathrm dy\\
\nonumber&+6D^{2}\int_0^1\int_0^xq(x-y)^{2}\mathrm dx\mathrm dy+8D^{2}q(0)^{2}\\
&+8D^{2}\int_0^1\int_0^xq'(x-y)^{2}\mathrm dx\mathrm dy.
\end{align}
Knowing that $q(x)=\gamma_y(x,1)$
\begin{align}
&\nonumber\int_0^1\int_0^x\gamma_y(x-y,1  )^2\mathrm dy\mathrm dx=\int_0^1\int_0^x\gamma_y(s,1  )^{2}\mathrm ds\mathrm dx\\
\nonumber\leq&\int_0^1\int_0^1\gamma_y(s,1  )^{2}\mathrm ds\mathrm dx\\
=&\int_0^1\gamma_y(x,1  )^{2}\mathrm dx, 
\end{align}
and
\begin{align}
\nonumber&\int_0^1\int_0^x\gamma_{xy}(x-y,1  )^2\mathrm dy\mathrm dx=\int_0^1\int_0^x\gamma_{xy}(s,1  )^2\mathrm ds\mathrm dx\\
\nonumber\leq&\int_0^1\int_0^1\gamma_{xy}(s,1  )^2\mathrm ds\mathrm dx\nonumber\\
=&\int_0^1\gamma_{xy}(x,1  )^2\mathrm dx, 
\end{align}
we have
\begin{align}
\nonumber S_1\leq&9+(18+22D^{2})\epsilon^2+4\int_0^1\int_0^xk(x,y)^{2}\mathrm dx\mathrm dy\\
\nonumber&+6\int_0^1\int_0^1\gamma(x,y)^2\mathrm dx\mathrm dy+8\int_0^1\int_0^1\gamma_x(x,y)^{2}\mathrm dx\mathrm dy\\
\nonumber&+6D^{2}\int_0^1\gamma_y(x,1  )^{2}\mathrm dx+8D^{2}k_y(1,1  )^{2}\\
&+8D^{2}\int_0^1\gamma_{xy}(x,1)^{2}\mathrm dx.\label{equ-S-9}
\end{align}

% 
% , $p(x,x  )=-\eta_{y}(0,1  )=-l_{y}(1,1)$ and $p_x(x,y  )=-\eta _{xy}(x-y,1  )$.
% Since
% \begin{align}
% \nonumber&\int_0^1\int_0^x\eta_{xy}(x-y,1  )^2\mathrm dy\mathrm dx\\
% \nonumber=&\int_0^1\int_0^x\eta_{xy}(s,1  )^2\mathrm ds\mathrm dx\\
% \nonumber\leq&\int_0^1\int_0^1\eta_{xy}(s,1  )^2\mathrm ds\mathrm dx\\
% =&\int_0^1\eta_{xy}(x,1D  )^2\mathrm dx, 
% \end{align}
% we have
Combining \eqref{equ-S-9}, Lemma \ref{lemma3}, and Lemma \ref{lemma4}, we obtain \eqref{equ-S-1-0}.

To prove \eqref{equ-S-2-0} in Proposition \ref{propo1}, we derive the following estimate of   the $L^2$ norm of $u$:
\begin{align}
\nonumber&\int_0^1u(x,t)^2\mathrm dx\\
\nonumber=&\int_0^1\bigg(\hat w(x,t)+\int_0^x\hat l(x,y)\hat  w(y,t)\mathrm dy\bigg)^2\mathrm dx\\
\nonumber\leq&2\int_0^1 \hat w(x,t)^2\mathrm dx+2\int_0^1\left(\int_0^x\hat l(x,y)\hat  w(y,t)\mathrm dy\right)^2\mathrm dx\\
\nonumber\leq&2\rVert\hat  w(t)\rVert^2+2\int_0^1\int_0^x\hat l(x,y)^2\mathrm dy\mathrm dx\rVert\hat  w(t)\rVert^2\\
\leq&2\left(1+\int_0^1\int_0^1\hat l(x,y)^2\mathrm dy\mathrm dx\right)\rVert\hat w(t)\rVert^2,\label{-equ-pro-u}
\end{align}
where  Cauchy-Schwarz inequality has been employed.  Then, the $L^2$ norm of $v$ is
\begin{align}
\nonumber&\int_0^1v(x,t)^2\mathrm dx\\
\nonumber=&\int_0^1\left(\hat z(x,t)+\int_0^1\hat \eta(x,y  )\hat  w(y,t)\mathrm dy\right.\\
\nonumber&+\left.D \int_0^x\hat p(x-y  )\hat z(y,t)\mathrm dy\right)^2\mathrm dx\\
\nonumber\leq&3\int_0^1\hat z(x,t)^2\mathrm dx+3\int_0^1\left(\int_0^1\hat \eta(x,y  )\hat  w(y,t)\mathrm dy\right)^2\mathrm dx\\
\nonumber&+3D^2\int_0^1\left(\int_0^x\hat p(x-y)\hat z(y,t)\mathrm dy\right)^2\mathrm dx\\
% \nonumber&\leq 3\int_0^1\int_0^1\hat \eta(x,y  )^2\mathrm dy\mathrm dx\rVert\hat   w\rVert^2\\
% &~~~+3\left(1+D^2 \int_0^1\int_0^x\hat \eta_y(x-y,1  )^2\mathrm dy\mathrm dx\right)\rVert\hat  z\rVert^2\\
\nonumber\leq &3\int_0^1\int_0^1\hat \eta(x,y  )^{2}\mathrm dy\mathrm dx\rVert\hat w(t)\rVert^2\\
&+3\left(1+D^2\int_0^1\int_0^x\hat p(x-y)^2\mathrm dy\mathrm dx\right)\rVert\hat  z(t)\rVert^2,\label{-equ-pro-v}
\end{align}
where we use Cauchy-Schwarz inequality.
% and $p(x,y  )=-\eta_y(x-y,1  )$. 
% Since
% \begin{align}
% \nonumber&\int_0^1\int_0^x\hat \eta_y(x-y,1  )^2\mathrm dy\mathrm dx\\
% \nonumber=&\int_0^1\int_0^x\hat \eta_y(s,1  )^{2}\mathrm ds\mathrm dx\leq\int_0^1\int_0^1\hat \eta_y(s,1  )^{2}\mathrm ds\mathrm dx\\
% =&\int_0^1\hat \eta_y(x,1  )^{2}\mathrm dx, 
% \end{align}
% we get 
% \begin{align}
% \nonumber\int_0^1v(x,t)^{2}\mathrm dx\leq &3\int_0^1\int_0^1\hat \eta(x,y  )^{2}\mathrm dy\mathrm dx\rVert\hat   w\rVert^2\\
% &+3\left(1+D^2 \int_0^1\hat \eta_y(x,1  )^{2}\mathrm dx\right)\rVert\hat  z\rVert^2.\label{norm_v}
%\end{align}

Next,  considering  the first derivative of $v(x,t)$ with respect to $x$, and exploiting   Cauchy-Schwarz inequality the following estimates are constructed
\begin{align}
&\nonumber\int_0^1v_x(x,t)^{2}\mathrm dx\\
\nonumber=&\int_0^1\bigg(\hat z_{x}(x,t)+\int_0^1\hat \eta_{x}(x,y  )\hat  w(y,t)\mathrm dy+D \hat p(0)\hat z(x,t)\\ 
\nonumber&+D \int_0^x\hat p'(x-y  )\hat z(y,t)\mathrm dy\bigg)^2\mathrm dx\\
\nonumber\leq&4\int_0^1\hat z_{x}(x,t)^{2}\mathrm dx+4\int_0^1\left(\int_0^1\hat \eta_{x}(x,y  )\hat  w(y,t)\mathrm dy\right)^2\mathrm dx\\
\nonumber&+4D^2\int_0^1(\hat p(0)\hat z(x,t))^2\mathrm dx\\
\nonumber&+4D^2\int_0^1\left(\int_0^x\hat p'(x-y  )\hat z(y,t)\mathrm dy\right)^2\mathrm dx\\
\nonumber\leq&4\int_0^1\int_0^1\hat \eta_x(x,y  )^{2}\mathrm dy\mathrm dx \rVert\hat   w(t)\rVert^2+4D^2\bigg(\hat p(0)^2\\
&+\int_0^1\int_0^x\hat p'(x-y  )^{2}\mathrm dy\mathrm dx\bigg)\rVert \hat z(t)\rVert^2+4\rVert\hat  z_{x}(t)\rVert^2.\label{-equ-pro-vx}
\end{align}
Then, from \eqref{-equ-pro-u}--\eqref{-equ-pro-vx}, we get 
\begin{align}
\nonumber S_2=&9+2\int_0^1\int_0^x\hat l(x,y)^2\mathrm dx\mathrm dy+3\int_0^1\int_0^1\hat \eta(x,y)^2\mathrm dx\mathrm dy\\
\nonumber&+4\int_0^1\int_0^1\hat \eta_x(x,y)^2\mathrm dx\mathrm dy+3D^{2}\int_0^1\int_0^x\hat p(x-y)^2\mathrm dx\mathrm dy\\
&+4D^{2}\hat p(0)^2+4D^{2}\int_0^1\int_0^x\hat p'(x-y)^2\mathrm dx\mathrm dy.\label{equ-pro-S2}
\end{align}

The next step is to prove the boundedness of $S_2$. First, we consider  the inverse gain kernels $\hat \eta(x,y)$, $\hat p(x)$ satisfy the following relations,
\begin{align}\label{equ-p-q}
\hat p(x)=&\hat q(x)+D\int_0^x\hat q(x-\xi)\hat p(\xi)\mathrm d\xi,\\
\nonumber\hat \eta(x,y)=&\hat \gamma(x,y)+D\int_y^1\gamma(x,\xi)\hat l(\xi,y)\mathrm d\xi\\
&+D\int_y^x\hat q(x-\xi)\hat \eta(\xi,y)\mathrm d\xi.
\end{align}
Knowing that the  inverse kernel satisfies the following conservative bound
\begin{align}
&\rVert \hat p\rVert_\infty\leq \rVert \hat q\rVert_\infty {\rm e}^{\rVert \hat q\rVert_\infty},\label{equ-pro-p}\\
&\rVert \hat \eta\rVert_\infty\leq (1+\rVert \hat l\rVert_\infty)\rVert \hat \gamma\rVert_\infty {\rm e}^{\rVert \hat q\rVert_\infty},\label{equ-pro-eta}
\end{align}
and  taking derivative of $\hat p$ and $\hat \eta$ with respect to $x$, the following holds
\begin{align}
\nonumber\hat p'(x)=&\hat q'(x)+D\hat q(0)\hat p(x)\\
&+D\int_y^x\hat q'(x-\xi)\hat p(\xi)\mathrm d\xi,\\
\nonumber\hat \eta_{x}(x,y)=&\hat \gamma_x(x,y)+D\int_y^1\gamma_{x}(x,\xi)\hat l(\xi,y)\mathrm d\xi+D\hat q(0)\hat \eta(x,y)\\
&+D\int_y^x\hat q'(x-\xi)\hat \eta(\xi,y)\mathrm d\xi.
\end{align}
Integrating \eqref{equ-pro-p} the following estimate can be derived
\begin{align}
\nonumber&\int_0^1\int_0^x\hat p'(x-y)^2\mathrm dy\mathrm dx\\
\nonumber=&\int_0^1\int_0^x\bigg(\hat q'(x-y)+D\hat q(0)\hat p(x-y)\\
\nonumber&+D\int_y^x\hat q'(x-\xi)\hat p(\xi-y)\mathrm d\xi\bigg)^{2}\mathrm dy\mathrm dx\\
\nonumber\leq&3\bigg(\int_0^1\int_0^x\hat q'(x-y)^{2}\mathrm dy\mathrm dx+D^{2}\rVert\hat q\rVert_\infty^2\rVert\hat p\rVert_\infty^2\\
&+D^2\rVert\hat p\rVert_\infty^2\int_0^1\int_0^x\hat q'(x-y)^{2}\mathrm dy\mathrm dx\bigg),\label{equ-pro-px}\\
\nonumber&\int_0^1\int_0^1\hat \eta_x(x,y)^2\mathrm dy\mathrm dx\\
\nonumber=&\int_0^1\int_0^1\bigg(\hat \gamma_x(x,y)+D\int_y^1\gamma_{x}(x,\xi)\hat l(\xi,y)\mathrm d\xi\\
\nonumber&+D\hat q(0)\hat \eta(x,y)+D\int_y^x\hat q'(x-\xi)\hat \eta(\xi,y)\mathrm d\xi\bigg)^{2}\mathrm dy\mathrm dx\\
\nonumber\leq&4(1+D^{2}\rVert\hat l\rVert_\infty^2)\int_0^1\int_0^1\hat \gamma_x(x,y)^{2}\mathrm dy\mathrm dx+4D^2\rVert\hat q\rVert_\infty^2\rVert\hat \eta\rVert_\infty^2\\
&+4D^2\rVert\hat \eta\rVert_\infty^2\int_0^1\int_0^x\hat q'(x-y)^{2}\mathrm dy\mathrm dx.\label{equ-pro-etax}
\end{align}
Finally, combining \eqref{equ-pro-S2} with \eqref{eq88}, \eqref{equ-pro-p}, \eqref{equ-pro-eta}, \eqref{equ-pro-etax},  Lemma \ref{lemma3} and Lemma \ref{lemma4}, one obtains that \eqref{equ-S-2-0} is bounded and completed the proof for the  Proposition \ref{propo1}.

%\textcolor{red}{Based on \eqref{Psi-final}, as $t\to \infty$, one can get the maximum $\epsilon$ as 
%\begin{align}
%\epsilon^*:=\frac{\sqrt{(B_u^{2}+B_v^{2}+B_{v_x}^{2})}}{\sqrt{S_2\kappa_3}}.
%\end{align}
%Moreover, as $t=0$, if $\rVert u(0)\rVert ^{2}+\rVert v(0)\rVert ^{2}+\rVert v_x(0)\rVert ^{2}\leq\zeta$, the radius $\zeta$ of the initial condition ball in the $L^2[0, 1]$ space, is given as 
%\begin{align}
%\zeta:=\frac{S_1}{S_2\kappa_2 }\bigg((B_u^{2}+B_v^2+B_{v_x}^2)-S_2{\kappa_3}\epsilon^2\bigg)>0.
%\end{align}}

\section*{Acknowledgments}

{The work of M. Diagne was funded by the NSF CAREER Award CMMI-2302030 and  the NSF grant CMMI-2222250. The work of M. Krstic was funded by the NSF grant ECCS-2151525 and the AFOSR grant FA9550-23-1-0535.}

\bibliographystyle{IEEEtranS}

\bibliography{reference}  
\end{document}